\documentclass[11pt]{amsart}
\usepackage{amsmath,amssymb}
\usepackage[all,cmtip]{xy}
\usepackage{mathtools}

\begin{document}

\theoremstyle{plain}
\newtheorem{thm}{Theorem}[section]
\newtheorem{lem}[thm]{Lemma}
\newtheorem{prop}[thm]{Proposition}
\newtheorem{cor}[thm]{Corollary}
\theoremstyle{definition}
\newtheorem{defn}[thm]{Definition}
\newtheorem{remark}{Remark}
\newtheorem{condition}[thm]{Condition}
\newtheorem{example}[thm]{Example}

\newcommand{\vfirstx}{\vec{x}^1}
\newcommand{\vsecondx}{\vec{x}^2}

\newcommand{\sqf}{square-free}
\newcommand{\Sqf}{Square-free}

\newcommand{\map}[1]{\tilde{#1}}
\newcommand{\per}[4]{\mathcal P_{#1,#2}(\vec{#3},#4)}

\newcommand{\persimple}[5]{\mathcal P_{#1,#2}(\vec{#3},\vec{#4},\vec{#5})}
\newcommand{\persimplelong}[5]{\persing{#1}{#3_{1}}{#4_{1}}{#5_{1}}\times \dots \times \persing{#1}{#3_{#2}}{#4_{#2}}{#5_{#2}} }

\newcommand{\multP}[3]{P_{f,\vec{#1},\vec{#2},\vec{#3}}}

\newcommand{\squaref}[3]{\mathcal S_{#1,#2}(#3)}
\newcommand{\persing}[4]{\mathcal P_{#1}(#2,#3,#4)}
\newcommand{\squarefsing}[2]{\mathcal S_{#1}(#2)}

\newcommand{\gmatrix}[3]{\begin{pmatrix}
 #1_{1,1} & #1_{1,2} & \cdots & #1_{1,#3} \\
 #1_{2,1} & #1_{2,2} & \cdots & #1_{2,#3} \\
 \vdots   & \vdots   & \ddots & \vdots && \\
 #1_{#2,1}& #1_{#2,2}& \cdots & #1_{#2,#3}
\end{pmatrix}}
\newcommand{\gvector}[2]{\begin{pmatrix}
#1_{1} \\
#1_{2}\\
\vdots \\
#1_{#2}
\end{pmatrix}}
\newcommand{\gvectorstart}[3]{\begin{pmatrix}
#1_{#2}\\
\vdots \\
#1_{#3}
\end{pmatrix}}
\newcommand{\partiald}[2]{\frac{\partial {#1}}{\partial #2}}
\newcommand{\ordinaryd}[2]{\frac{\d {#1}}{\d #2}}
\newcommand{\polring}[3]{#1[#3_1,\dots,#3_{#2}]}
\newcommand{\polfield}[3]{#1(#3_1,\dots,#3_{#2})}
\newcommand{\btwn}[3]{{#1}\;,{#2} \leq {#1} \leq {#3}}

\newcommand{\av}{a}
\newcommand{\bv}{b}
\newcommand{\cv}{c}

\newcommand{\uv}{u}

\newcommand{\bbox}{\operatorname{Box}}

\newcommand{\rank}{\operatorname{Rank}}
\newcommand{\spanv}{\operatorname{Span}}
\newcommand{\image}{\operatorname{Image}}
\newcommand{\chr}{\operatorname{Char}}

\numberwithin{equation}{section}

\newcommand{\Z}{{\mathbb Z}} 
\newcommand{\Q}{{\mathbb Q}}
\newcommand{\R}{{\mathbb R}}
\newcommand{\C}{{\mathbb C}}
\newcommand{\N}{{\mathbb N}}
\newcommand{\FF}{{\mathbb F}}
\newcommand{\fe}{\overline{\mathbb F}}
\newcommand{\fq}{\mathbb{F}_q}
\newcommand{\feq}{\overline{\mathbb F}_q}

\newcommand{\rmk}[1]{\footnote{{\bf Comment:} #1}}

\renewcommand{\mod}{\;\operatorname{mod}}
\newcommand{\ord}{\operatorname{ord}}
\newcommand{\TT}{\mathbb{T}}
\renewcommand{\i}{{\mathrm{i}}}
\renewcommand{\d}{{\mathrm{d}}}
\newcommand{\HH}{\mathbb H}
\newcommand{\Vol}{\operatorname{vol}}
\newcommand{\area}{\operatorname{area}}
\newcommand{\tr}{\operatorname{tr}}
\newcommand{\norm}{\mathcal N} 
\newcommand{\intinf}{\int_{-\infty}^\infty}
\newcommand{\ave}[1]{\left\langle#1\right\rangle} 
\newcommand{\Var}{\operatorname{Var}}
\newcommand{\Prob}{\operatorname{Prob}}
\newcommand{\sym}{\operatorname{Sym}}
\newcommand{\disc}{\operatorname{disc}}
\newcommand{\CA}{{\mathcal C}_A}
\newcommand{\cond}{\operatorname{cond}} 
\newcommand{\lcm}{\operatorname{lcm}}
\newcommand{\Kl}{\operatorname{Kl}} 
\newcommand{\leg}[2]{\left( \frac{#1}{#2} \right)}  

\newcommand{\sumstar}{\sideset \and^{*} \to \sum}

\newcommand{\LL}{\mathcal L} 
\newcommand{\sumf}{\sum^\flat}
\newcommand{\Hgev}{\mathcal H_{2g+2,q}}
\newcommand{\USp}{\operatorname{USp}}
\newcommand{\conv}{*}
\newcommand{\dist} {\operatorname{dist}}
\newcommand{\CF}{c_0} 
\newcommand{\kerp}{\mathcal K}

\newcommand{\fs}{\mathfrak S}
\newcommand{\rest}{\operatorname{Res}} 
\newcommand{\af}{\mathbb A} 
\newcommand{\Ht}{\operatorname{Ht}}
\newcommand{\monic}{\mathcal M}

\title
[Square-free values of polynomials in linear sparse sets] {Square-free values of multivariate polynomials over function fields in linear sparse sets}
\author{Shai Rosenberg}
\address{Raymond and Beverly Sackler School of Mathematical Sciences,
Tel Aviv University, Tel Aviv 69978, Israel}
\email{shairos1@mail.tau.ac.il}

\maketitle

\begin{abstract}
Let $f \in \fq[t][x]$ be a \sqf{} polynomial where $\fq$ is a field of $q$ elements. We view $f$ as a polynomial in the variable $x$ with coefficients in the ring $\fq[t]$. We study \sqf{} values of $f$ in sparse subsets of $\fq[t]$ which are given by a linear condition. The motivation for our study is an analogue problem of representing \sqf{} integers by integer polynomials, where it is conjectured that setting aside some simple exceptional cases, a \sqf{} polynomial $f \in \Z[x]$ takes infinitely many square-free values. Let $\kappa \in \N$ be co-prime to $q$, and let $\gamma_1,\dots,\gamma_{\kappa-1},\gamma_{\kappa+1}\dots,\gamma_m \in \fq$. A consequence of the main result we show, is that if $q$ is sufficiently large with respect to $\deg_x f, \deg_t f$ and $m$, then there exist $\gamma_0,\gamma_{\kappa} \in \fq$ such that $f\left (t,\sum_{i=0}^m \gamma_i t^i \right )$ is \sqf{}. Moreover, as $q \to \infty$, the last is true for almost all $\gamma_0,\gamma_{\kappa} \in \fq$. The main result shows that a similar result holds also for other cases. We then generalize the results to multivariate polynomials.
\end{abstract}

\section{Introduction}

Let $f \in \fq[t][x]$ where $\fq$ is a field of $q$ elements and $p=\chr(\fq)$. We consider $f$ as a univariate polynomial in $x$ where its coefficients lay in the ring $\fq[t]$. The result of substituting the variable $x$ with an element in the base ring $\fq[t]$, is a polynomial in $\fq[t]$, i.e. for any $\uv \in \fq[t]$, $f(t,\uv(t)) \in \fq[t]$. A polynomial is said to be \sqf{} if it does not have a nonconstant square divisor. If there exists $\uv \in \fq[t]$ such that $f(t,\uv(t))$ is \sqf{}, then $f$ is said to have a \sqf{} value at $\uv$. Given a polynomial $f \in \fq[t][x]$, we are motivated by the question of whether $f$ has \sqf{} values. Moreover, we would like to estimate the number of \sqf{} values of $f$ and show that it is large in some sense. If $f$ is not \sqf{} then we can not expect $f$ to have many \sqf{} values. This is because if $g^2$ divides $f$ where $g \in \fq[t][x]$ is a nonconstant polynomial, then for any $\uv \in \fq[t]$ such that $\deg g(t,\uv(t)) > 0$, $g(t,\uv(t))^2$ is a nonconstant square factor of $f(t,\uv(t))$.

Hence we require $f$ to be a \sqf{} polynomial. A natural question is then whether this condition is sufficient, i.e. whether a \sqf{} polynomial always has \sqf{} values.

This question may be viewed as a function field analogue of a known open conjecture which concerns polynomials over $\mathbb{Z}$. In the analogue question, instead of considering $f$ as a polynomial over $\fq[t]$, $f$ is considered as a polynomial over $\Z$. The conjecture is that setting aside some simple exceptional cases, given a \sqf{} polynomial $f \in \Z[x]$ there are infinitely many $n \in \N$ such that $f(n)$ is a \sqf{} number, and moreover, the set of \sqf{} values of $f$ has positive density.

The case where $f$ is quadratic was solved by Ricci \cite{Ricci}. For the case where $f$ is cubic, Erd\"{o}s \cite{Erdos} showed that there are infinitely many \sqf{} values, and Hooley \cite{Hooley 1967} showed that the set of \sqf{} values has positive density. Granville \cite{Granville} showed that assuming the ABC conjecture the problem is completely settled.

Returning to the question over function fields, a quantitative statement of the question is to estimate the number of polynomials $\uv \in \fq[t]$ such that $f(t,\uv(t))$ is \sqf{}. This can be asked in the context of two limits. One is to fix a polynomial $f$ and count number of $\uv \in \fq[t]$ of degree $m$ such that $f(t,\uv(t))$ is \sqf{} while $m$ tends to infinity. The other limit is to fix $m$ and count the number of $\uv \in \fq[t]$ of degree $m$ such that $f(t,\uv(t))$ is \sqf{} while $q$ tends to infinity.

For any field $\FF$, let
\begin{equation}
\monic_m(\FF)=\{\uv\in \FF[t]: \deg \uv=m,\uv \mbox{ monic} \} \;,
\end{equation}
so that $\#\monic_m(\fq)=q^m$. Defining
\begin{equation}
\squarefsing{\FF}{f} = \{\uv \in \FF[t] : f(t,\uv(t)) \mbox{ is square-free}
\} \;,
\end{equation}
in \cite{Rudnick} Rudnick studied the frequency

\begin{equation}
\frac {\#(\squarefsing{\fq}{f}\bigcap\monic_m(\fq))}{\#\monic_m(\fq)}
\end{equation}
and showed that, assuming $f \in \fq[t][x]$ is separable with \sqf{} content, as $q\to \infty$,
\begin{equation}\label{equation introduction - count squarefree as q goes to infinity}
\frac{(\#\squarefsing{\fq}{f}\bigcap\monic_m(\fq))}{ \#\monic_m(\fq)} = 1 +O\left (\frac {(m\deg_x
f+\deg_t f)\deg_x f}{q} \right )\;,
\end{equation}
where the implied constant is absolute. In the estimate above $f$ is not assumed to be fixed. Indeed, fixing $f$ makes little sense as the base field of $\fq$ may change as $q \to \infty$. However the estimate depends only on $m$ and a bound on the degree of $f$, so $f$ may vary while $q \to \infty$ as long as its degree remains bounded.

In particular, Eq.~\ref{equation introduction - count squarefree as q goes to infinity} shows that if $q$ is sufficiently large w.r.t. $m$, $\deg_x f$ and $\deg_t f$, then there exists $\uv \in \monic_m(\fq)$ such that $f(t,\uv(t))$ is \sqf{}. Moreover, Eq.~\ref{equation introduction - count squarefree as q goes to infinity} shows that in some sense this is true for almost all $\uv \in \monic_m(\fq)$.

The key tool in \cite{Rudnick} is the use of the discriminant of $f(t,\uv(t))$ in order to tell whether $f(t,\uv(t))$ is \sqf{}. If $f(t,\uv(t))$ is not \sqf{}, the discriminant of $f(t,\uv(t))$ vanishes. The last can be translated into a polynomial condition on the coefficients of $\uv$. Hence the problem can be converted to an algebraic statement about the number of zeros of a polynomial. It may be interesting to note that this tool seems unavailable in the analogue question over $\Z$.

In this note we extend the results of Rudnick by considering a stronger version of the question. Instead of asking whether there exists a polynomial $\uv \in \fq[t]$ such that $f(t,\uv(t))$ is \sqf{} where $\uv$ is a monic polynomial of degree $m$, we will ask whether there exists such polynomial $\uv$ of a specific form, for example $\uv=t^m+\beta$ where $\beta \in \fq$. Throughout this note, when saying that a polynomial $\tilde{\uv}$ is obtained by perturbing one or more coefficients of a polynomial $\uv$, we mean that $\tilde{\uv}$ is obtained by changing only those coefficients of $\uv$ while leaving the other coefficients of $\uv$ unchanged. For example, $t^m+1$ is obtained by perturbing the free coefficient of $t^m$. Let $\kappa \in \N$, such that $1 \leq \kappa \leq m$ and $\kappa \not = 0 \mod p$. Consider an arbitrary polynomial $\uv \in \fq[t]$, $\uv(t)=\sum_{i=1,i \not = \kappa}^{m}\gamma_it^i$, where $\gamma_1,\dots,\gamma_{\kappa-1},\gamma_{\kappa+1},\dots,\gamma_m \in \fq$. We will show that as $q \to \infty$, for almost all $\gamma_0,\gamma_{\kappa} \in \fq$, $f(t,\sum_{i=0}^m\gamma_it^i)$ is \sqf{}. Namely, by perturbing two of the coefficients of $\uv$ we obtain \sqf{} values of $f$. The last is a special case of the main theorem of this note, in which we also consider similar sparse sets, more general than the set which corresponds to perturbations of two of the coefficients of a polynomial $\uv$.

As in \cite{Rudnick} we do a similar use of the discriminant in order to translate the problem to an algebraic theorem which holds for any field. In Section~\ref{section the discriminant and constant assignments over a general field} we describe how the discriminant may be used in showing the existence of \sqf{} values. In this section we discuss first the case of assigning constants from the base field. This is the case where $f \in \polring{\FF[t]}{d}{x}$ is a multivariate polynomial, and we ask whether there exist $\beta_1,\dots,\beta_d \in \FF$ such that $f(t,\beta_1,\dots,\beta_d)$ is \sqf{}. In Section~\ref{section the discriminant and constant assignments over a general field} we also extend the use of the discriminant properties, and in particular the fact that the expression for the discriminant is independent of the base field. By that we prove an algebraic lemma which holds over a general field $\FF$ in the case of constants assignments. The algebraic lemma which we present for constant assignments will also be used later, when we handle non-constant assignments.

The main result we show, provides an estimate of the number of \sqf{} values of $f$ in sparse subsets of $\fq[t]$ which are given by a linear condition of a certain kind. We now describe what these sparse sets are, and introduce the notations we use for defining them.

Let $\FF$ be a field. Let $\av,\bv,\cv \in \FF[t]$. Define

$$
\persing{\FF}{\av}{\bv}{\cv}:=\left \{\av\beta_1+\bv\beta_2+\cv: \beta_1,\beta_2 \in \FF \right \}.
$$

For example, if $\cv(t)=\sum_{i=0}^m \gamma_i t^i$ where $\gamma_1,\dots,\gamma_m \in \FF$, then

$$
\persing{\FF}{1}{0}{\cv} = \left \{\sum_{i=1}^m \gamma_i t^m+\beta_1 : \beta_1 \in \FF \right \}.
$$

In this case $\persing{\FF}{1}{0}{\cv}$ is the set of all polynomials in $\FF[t]$ that one gets by perturbing the free coefficient of $\cv(t)$. Similarly, $\persing{\FF}{1}{t}{\cv}$ denotes the set of polynomials in $\FF[t]$ that one gets by perturbing the coefficient of $t$ and the free coefficient of the polynomial $\cv(t)$.

$$
\persing{\FF}{1}{t}{\cv} = \left \{\sum_{i=2}^m \gamma_i t^m+\beta_2 t+\beta_1: \beta_1, \beta_2 \in \FF \right \}
$$

In general, if $\av,\bv \in \{1,t,t^2,\dots\}$ then $\persing{\FF}{\av}{\bv}{\cv}$ denotes the polynomials obtained by perturbing two coefficients of $\cv$. If $\bv=0$ and $\av \in \{1,t,t^2,\dots\}$ then $\persing{\FF}{\av}{0}{\cv}$ corresponds to perturbing one coefficient of $\cv(t)$.

In the more general case where $\av, \bv$ are not necessarily in $\{1,t,t^2,\dots\}$, $\persing{\FF}{\av}{\bv}{\cv}$ is a subset of $\FF[t]$. In the case of a finite field $\fq$, the size of $\persing{\fq}{\av}{\bv}{\cv}$ satisfies $\#\persing{\fq}{\av}{\bv}{\cv} \leq q^2$.

We are interested in finding condition on $\av$, $\bv$ and $\cv$ that guarantee the existence of \sqf{} values of $f$, when $\uv$ is restricted to the set $\persing{\fq}{\av}{\bv}{\cv}$, provided that $q$ is sufficiently large. Moreover, we will see that for such $\av,\bv,\cv$, as $q \to \infty$ $f$ has a \sqf{} value at almost all the elements of $\persing{\fq}{\av}{\bv}{\cv}$, that is:
$$
   \frac{\#(\squarefsing{\fq}{f} \bigcap \persing{\fq}{\av}{\bv}{\cv})}{\#\persing{\fq}{\av}{\bv}{\cv}} = 1 + O \left (\frac{1}{q} \right )\;,\quad \mbox{as } q\to \infty.
$$

We have
$$
\frac{\#\persing{\fq}{\av}{\bv}{\cv}}{\#\monic_m(\fq)} \leq \frac{q^2}{q^m}.
$$
Assuming $m \geq 3$ and $\deg \av, \deg \bv, \deg \cv \leq m$, then while keeping $m$ fixed
$$
\lim_{q \rightarrow \infty} \frac{\#\persing{\fq}{\av}{\bv}{\cv}}{\#\monic_m(\fq)} =0.
$$

This shows that if $\av,\bv,\cv \in \fq[t]$ are such that $\persing{\fq}{\av}{\bv}{\cv} \subseteq \monic_m(\fq)$, then $\persing{\fq}{\av}{\bv}{\cv}$ is sparse with respect to $\monic_m(\fq)$ in the limit $q \to \infty$, so indeed claiming that there exists $\uv \in \persing{\fq}{\av}{\bv}{\cv}$ such that $f(t,\uv(t))$ is \sqf{} for a given triple $\av,\bv,\cv \in \fq[t]$ is stronger than claiming that there exists such $\uv \in \monic_m(\fq)$. We note that for some triples $\av,\bv,\cv$ $\persing{\FF}{\av}{\bv}{\cv}$ may not be a subset of $\monic_m(\fq)$. We allow such choice of $\av$, $\bv$, $\cv$ as well.

The main result in the case where $f$ is a univariate polynomial over $\fq[t]$ is presented in Section~\ref{section the main results}, where we introduce the main results of this note. This result is proved in Section~\ref{section Proof of Theorem single variable}. In Section~\ref{section simple gen} we state and prove a generalization of this result to the case where $f$ is a multivariate polynomial $f \in \polring{\FF[t]}{d}{x}$.

\section*{Acknowledgments}

This work is part of the author's M.Sc. thesis, written under the supervision of Ze\'ev Rudnick at Tel Aviv University. Partially supported by the Israel Science Foundation (grant No. 1083/10). I would like to thank Prof. Ze\'ev Rudnick for his guidance.

\subsection{Definitions and notations}\label{section Definitions and notations}
\begin{enumerate}
\item $\FF$ denotes a general field. $\fq$ denotes a finite field of $q$ elements. The characteristic of $\FF$ is denoted by $p$ or $\chr(\FF)$. We also use $L,K$ for general fields, in case that more than one field is considered.
\item $\FF[t][x]$ denotes the ring of polynomials in $t$ and $x$ over $\FF$. By analogy with the ring of integers, we consider $f$ as a univariate polynomial in $x$ over the ring $\FF[t]$, hence the notation. Similarly for multivariate polynomials over $\FF[t]$ we use the notation $\polring{\FF[t]}{d}{x}$.
\item Let $D$ be a unique factorization domain. An element $r \in D$ is \textbf{\sqf{}} if every $s \in D$ such that $s^2|r$ is invertible. Two elements $v_1,v_2 \in D$ are called \textbf{associated} if there exists an invertible $\alpha \in D$ such that $v_1=\alpha v_2$. Let $r=\prod_{i=1}^kr_i$ be a factorization of $r$ into irreducible factors. Then $r$ is \sqf{} if and only if for every $i,j$ such that $i \not=j$, $r_i$ and $r_j$ are not associated. For our purposes $D$ will be a polynomial ring. In cases where $r$ can be considered as an element in two unique factorization domains $D,\tilde{D}$ where $\tilde{D} \supset D$, we specify in which ring we assume $r$ is \sqf{} by saying that $r$ is \sqf{} in $R$ or $r$ \sqf{} in $\tilde{D}$. The same meaning holds when saying that $r$ is irreducible in $D$, or $r$ is irreducible in $\tilde{D}$, and also when saying that $d \in D$ divides $r \in D$ in $D$ or $d$ divides $r$ in $\tilde{D}$.
\item Let $\FF$ be a field - we denote by $\overline{\FF}$ an algebraic closure of $\FF$, also $\overline{\FF(x)}$ denotes an algebraic closure of $\FF(x)$ etc. We also assume that $\overline{\FF(x)}$ is chosen such that it contains $\overline{\FF}$.
\item For a vector $(a_1,a_2,\dots,a_n) \in \FF[t]^n$, define $$\|(a_1,\dots,a_n)\|:=\max\{\deg a_1,\dots,\deg a_n\}.$$

\item A polynomial $f \in \FF[x]$ is \textbf{separable} if all its roots in an algebraic closure of $\FF$ are distinct. If $f \in \polring{\FF}{d}{x}$ is a multivariate polynomial, and $i \in \N, 1 \leq i \leq d$, then $f$ is separable in $x_i$ if $f$ is separable when considering $f$ as a univariate polynomial in the variable $x_i$ over the field $\FF(x_1,\dots,x_{i-1},x_{i+1},\dots,x_d)$.
\item Let $D$ be an integral domain. A polynomial $f \in D[x]$ is \textbf{primitive} if the only elements in $D$ that divide all the coefficients of $f$ are the invertible elements in $D$.
\item A field $\FF$ is \textbf{perfect} if either it has characteristic $0$, or when $p>0$, for any $c \in \FF$, $c^{\frac{1}{p}} \in \FF$ holds.

\item Let $R_1,R_2$ be rings. Let $R$ be a subring of $R_1$ and $R_2$. A \textbf{$R$- homomorphism} is a homomorphism $R_1 \rightarrow R_2$ such that $r \mapsto r$ for every $r \in R$.
\item For a polynomial $f \in \polring{\FF[t]}{d}{x}$, $\deg_t f$ denotes the degree of $f$ in the variable $t$, similarly, for $i \in \N, 1 \leq i \leq d$, $\deg_{x_i} f$ denotes the degree of $f$ in $x_i$. $\deg f$ denotes the total degree of $f$ in all variables $t,x_1,\dots,x_d$. $\deg_{\vec{x}} f$ denotes the total degree of $f$ when considered as a polynomial in the variables $x_1,\dots,x_d$ over $\FF[t]$.

\item Let $f \in \FF[x]$. $\Delta f$ denotes the discriminant of $f$. Let $\gamma_k \in \FF$ be the leading coefficient of $f$. Then $\Delta f = \gamma_k^{2k-2}\prod_{i < j}(r_i-r_j)^2$ where $r_1,\dots,r_k$ are the roots of $f$ in $\overline{\FF}$. $D^k$ denotes the expression for the discriminant in terms of the coefficients of $f$. For example, if $f=\gamma_2x^2+\gamma_1x+\gamma_0$ then $D^k (f)$ = $\gamma_1^2-4\gamma_2\gamma_0$. If $f \in \polring{\FF[t]}{d}{x}$ is a multivariate polynomial, then $\Delta_t,\Delta_{x_i}$ and $D^k_{t}, D^k_{x_i}$ denote the corresponding notations when considering $f$ as a univariate polynomial in $t$ or $x_i$ respectively.
\item Let $\av,\bv,\cv \in \FF[t]$. Let
$$
\persing{\FF}{\av}{\bv}{\cv}:=\{\av\beta_1+\bv\beta_2+\cv: \beta_1,\beta_2 \in \FF\}.
$$

\item Given a polynomial $f \in \FF[t][x]$, let
$$
\squarefsing{\FF}{f} := \{\uv \in \FF[t] : f(t,\uv(t)) \mbox{ is square-free}
\}.
$$
For a multivariate polynomial $f \in \polring{\FF[t]}{d}{x}$ the corresponding notation is
$$
\squaref{\FF}{d}{f} := \{\vec{\uv{}} \in \FF[t]^d : f(t,\uv_1(t),\dots,\uv_d(t)) \mbox{ is square-free}\}.
$$
\item We denote the set of monic polynomials of degree $m$ by $\monic_m(\FF)$, namely $\monic_m(\FF)=\{\uv\in \FF[t]: \deg \uv=m,\uv \mbox{ monic} \}$. In the case where $\FF=\fq$ we abbreviate and write $\monic_{m}$.

\end{enumerate}

\section{The main results}\label{section the main results}
\subsection{\Sqf{} values of a univariate polynomial}

We start by stating the main theorem for univariate polynomials in its general form, and then showing a few specific examples which are special cases of the general theorem. Recall that for polynomials $\av,\bv,\cv \in \FF[t]$, we define
$$\|(\av,\bv,\cv)\| = \max \{\deg \av,\deg \bv, \deg \cv\}.$$

\begin{thm}\label{thm perturbations of a single variable polynomial finit
field}
Let $f \in \fq[t][x]$ be a \sqf{} polynomial. Let $\av,\bv,\cv \in \fq[t]$ such that $\gcd(\av,\bv)=1$. Let $N \in \N$. Assume $\deg_x f,\deg_t f, \|(\av,\bv,\cv)\| \leq N.$
Assume that at least one of the following holds

\begin{enumerate}
\item \label{part single variable theorem large char}$p > C(N)$ where $C(N)$ is a constant which depends only on $N$.
\item\label{part single variable theorem arbitrary char}$\frac{\bv}{\av} \not \in \FF(t^p)$ where $\av \not = 0$.

Then while $N$ remains fixed, we have:
\begin{equation}\label{equation main theorem single variable estimate}
   \frac{\#(\squarefsing{\fq}{f} \bigcap \persing{\fq}{\av}{\bv}{\cv})}{\#\persing{\fq}{\av}{\bv}{\cv}} = 1 + O\left (\frac{1}{q}\right)\;,\quad \mbox{as } q\to \infty.
\end{equation}
In particular, if $q$ is sufficiently large with respect to $N$ then there exist $\beta_1,\beta_2 \in \fq$ such that $f(t,\cv(t)+\av(t)\beta_1+\bv(t)\beta_2)$ is \sqf{}.
\end{enumerate}
\end{thm}

If $q$ is taken to be large then $p=\chr (\fq)$ may still remain small. For example if we fix a prime number $p$, then $\fq$ may be some algebraic extension of $\mathbb{F}_p$ of large degree. On the other hand, $p$ and $q$ may both be large, for example if we take $q=p$ and consider $\mathbb{F}_p$ where $p \to \infty$. (\ref{part thm sigle variable - albgebraic - single variable perturbation}) in Theorem~\ref{thm perturbations of a single variable polynomial finit field} can be viewed as the case where $\chr(\fq)$ is large. Considering $\mathbb{F}_p$ where $p \to \infty$ is an example of this case. (\ref{part thm sigle variable - albgebraic - two variables perturbation}) provides the conditions on $\av,\bv$ in the case where $\fq$ is a field with an arbitrary positive characteristic. We introduce two examples of Theorem~\ref{thm perturbations of a single variable polynomial finit field}, one for each of the two cases.

\begin{example}\label{example single variable - one variable perturbation}
Let $f \in \fq[t][x]$ be \sqf{}. Let $\cv \in \fq[t]$, given by $\cv(t)=\sum_{i=0}^m\gamma_i t^i$. If we take $\bv=0,\av=1$ then $\gcd(\bv,\av)=1$. Hence by Theorem~\ref{thm perturbations of a single variable polynomial finit field} if $q$ and $p$ are sufficiently large with respect to $N$, then there exists $\beta_1 \in \fq$ such that $f(t,\cv(t)+\beta_1)$ is \sqf{}, where $\cv(t)+\beta_1$ is a polynomial obtained by a perturbation of the free coefficient of $\cv$.
\end{example}

\begin{example}\label{example single variable - two variable perturbation}
Let $f \in \fq[t][x]$ be \sqf{}. Let $\kappa \in \mathbb{N}$ such that $\kappa \not = 0 \mod p$. Let $\cv \in \fq[t]$, given by $\cv(t)=\sum_{i=0}^m\gamma_i t^i$. If we take $\bv=t^\kappa,\av=1$ then $\gcd(\bv,\av)=1$. Also $\frac{\bv}{\av}=t^{\kappa} \not \in \fq(t^p)$. This shows that (\ref{part single variable theorem arbitrary char}) in Theorem~\ref{thm perturbations of a single variable polynomial finit field} holds. Hence by the same theorem if $q$ is sufficiently large with respect to $N$, then there exist $\beta_1,\beta_2$ such that $f(t,\cv(t)+\beta_2t^\kappa+\beta_1)$ is \sqf{}, where $\cv(t)+\beta_2t^{\kappa}+\beta_1$ is a polynomial obtained by a perturbation the free coefficient and the coefficient of $t^{\kappa}$ of $\cv$. In particular, in the case where $\kappa=1$, a \sqf{} value of $f$ is obtained by perturbing the first two coefficients of $\cv$.
\end{example}

As the first example above shows, the large characteristic case allows us to take one of $\av$ or $\bv$ to be $0$, while the other be $1$. This is because $\gcd(1,0)=1$, hence (\ref{part single variable theorem large char}) of Theorem~\ref{thm perturbations of a single variable polynomial finit
field} holds for this choice of $\av,\bv$. However, for an arbitrary positive characteristic both $\av$ and $\bv$ are non-zero as this is required in (\ref{part thm sigle variable - albgebraic - two variables     perturbation}) of Theorem~\ref{thm perturbations of a single variable polynomial finit field}. Hence in the case of a large characteristic it is sufficient to perturb a single coefficient of $\cv$ in order to obtain a \sqf{} value of $f$, while in the case of an arbitrary positive characteristic it might be necessary to perturb two coefficients of $\cv$.

The following two examples show why the assumption that $\gcd(\av,\bv)=1$ is required in Theorem~\ref{thm perturbations of a single variable polynomial finit field}, and why the assumption that $\frac{\bv}{\av} \not \in \fq(t^p)$ is required in (\ref{part thm sigle variable - albgebraic - two variables     perturbation}) of Theorem~\ref{thm perturbations of a single variable polynomial finit field}.

\begin{example}\label{example single gcd is needed}
Let $\av=t,\bv=t^2,\cv=0$. In this case $\gcd(\av,\bv)=t$. Let $f=x(x+t)$. Then $f$ is \sqf{} but
$$
f(t,t\beta_1+t^2\beta_2)=(t\beta_1+t^2\beta_2)(t\beta_1+t^2\beta_2+t)=t^2(\beta_1+t\beta_2)(\beta_1+t\beta_2+1)
$$
which is divisible by $t^2$. Hence $f(t,t\beta_1+t^2\beta_2)$ is not \sqf{} for any choice of $\beta_1,\beta_2 \in \fq$.
\end{example}

\begin{example}
Let $\av=1,\bv=t^p,\cv=t$. Let $f=x-t$, which is irreducible and in particular \sqf{} but
$$
f(t,\beta_1+t^p\beta_2+t)=\beta_1+t^p\beta_2=\left (\beta_1^{\frac{1}{p}}+t\beta_2^{\frac{1}{p}}\right)^p.
$$
Since $\fq$ is perfect $\beta_1^{\frac{1}{p}},\beta_2^{\frac{1}{p}} \in \fq$. Hence $f(t,\beta_1+t^p\beta_2+t)$ is not \sqf{} for any $\beta_1,\beta_2 \in \fq$.
\end{example}

\begin{remark}\label{remark meaning of necessary conditions}
Let $\av,\bv,\cv \in \fq[t]$ such that $\gcd(\av,\bv) \not = 1$. By claiming that the condition $\gcd(\av,\bv)=1$ is required in Theorem~\ref{thm perturbations of a single variable polynomial finit field} we do not mean that Eq.~\ref{equation main theorem single variable estimate} in Theorem~\ref{thm perturbations of a single variable polynomial finit field} cannot hold for a specific choice of a \sqf{} $f$.

Instead, by claiming that the condition is required we mean that if $$\gcd(\av,\bv) \not =1$$
then there exists a \sqf{} polynomial $f$ such that $f(t,\beta_1\av(t)+\beta_2\bv(t)+\cv(t))$ is not \sqf{} for any $\beta_1,\beta_2 \in \fq$.

The same meaning applies also when we claim that the condition $\frac{\bv}{\av} \not \in \fq(t^p)$ is required in (\ref{part thm sigle variable - albgebraic - two variables     perturbation}) of Theorem~\ref{thm perturbations of a single variable polynomial finit field}.

We do not place any restrictions on $f$ in Theorem~\ref{thm perturbations of a single variable polynomial finit field} other than that it should be \sqf{} and that its degree remains bounded while $q \to \infty$. The conditions insure the existence of \sqf{} values for any such $f$.
\end{remark}

\subsection{Square-free values of a multivariate polynomial}\label{subsection sqf values of a multivariate polynomial}
In Section~\ref{section simple gen} we state and prove Theorem~\ref{thm simple gen perturbations of a polynomial finit field} which is a generalization of Theorem~\ref{thm perturbations of a single variable polynomial finit field} that holds for multivariate polynomials.

We use Theorem~\ref{thm simple gen perturbations of a polynomial finit field} in order to estimate the number of \sqf{} values of a multivariate polynomial $f$ at the set $\monic_{m_1}\times \dots \times \monic_{m_d}$ where $m_1,\dots,m_d \in \N$, $\deg_t f,\deg_{\vec{x}} f$ are fixed and $q \to \infty$. This result is stated in Corollary~\ref{cor sqf values in of monic polyomials}. Corollary~\ref{cor sqf values in of monic polyomials} generalizes the result in \cite{Rudnick} to the case of multivariate polynomials. In Section~\ref{section simple gen} we will show that it follows from Theorem~\ref{thm simple gen perturbations of a polynomial finit field}.

\begin{cor}[\sqf{} values of multivariate polynomials over a finite filed]\label{cor sqf values in of monic polyomials}
Let $f \in \polring{\fq[t]}{d}{x}$ be a \sqf{} polynomial. Let $m_1,\dots,m_d \in \N$. Let $N \in \N$. Assume $\deg_{\vec{x}} f,\deg_t f, m_1,\dots,m_d \leq N$ and $2 \leq m_1,\dots,m_d$. Then while $N$ remains fixed, the following holds:

\begin{equation}\label{equation multivariate monic theorem single variable estimate}
   \frac{\#(\squaref{\fq}{d}{f} \bigcap (\monic_{m_1} \times \dots\times \monic_{m_d}))}{\#\monic_{m_1} \times \dots\times \monic_{m_d}}= 1 + O \left (\frac{1}{q} \right )\;,\quad \mbox{as } q\to \infty.
\end{equation}
In particular, if $q$ is sufficiently large with respect to $N$ there exist $\uv_1 \in \monic_{m_1},\dots,\uv_d \in \monic_{m_d}$ such that $f(t,\uv_1(t),\dots,\uv_d(t))$ is \sqf{}.
\end{cor}

An estimate in the case where $q$ is fixed and the degrees of $\uv_1,\dots,\uv_d$ are allowed to grow was proved by Poonen in \cite{Poonen Duke}. Let $f \in \polring{\fq[t]}{d}{x}$ be a polynomial which is \sqf{} as an element of $\polring{K}{d}{x}$, where $K$ denotes the field of fractions of $\fq[t]$. Let $B_1,\dots,B_d \in \N$ and define
$$\bbox=\bbox(B_1,\dots,B_d):=\{(\uv_1,\dots,\uv_d) \in \fq[t]^d: \deg \uv_i \leq B_i \mbox{ for all }i\}.$$
For a prime $\mathfrak{p}$ in $\fq[t]$ let $c_{\mathfrak{p}}$ denote then number of $x \in (\fq[t]/{\mathfrak{p}^2})^d$ satisfying $f(x)=0$ in $\fq[t]/\mathfrak{p}^2$. Poonen showed that
$$
\lim_{B_1,\dots,B_d \to \infty}\frac{\#(\bbox \bigcap \squaref{\fq}{d}{f})}{\# \bbox}=\prod_{\mathfrak{p} \mbox{ prime}}\left (1-\frac{c_{\mathfrak{p}}}{|\mathfrak{p}|^{2d}}\right ).
$$

Theorem~\ref{thm simple gen perturbations of a polynomial finit field}, which generalizes Theorem~\ref{thm perturbations of a single variable polynomial finit field} to multivariate polynomials, will be stated in Section~\ref{section simple gen}. Here we only introduce an example which is a specific case of Theorem~\ref{thm simple gen perturbations of a polynomial finit field}.

\begin{example}\label{example simple perturbation generalization}
Let $f \in \polring{\fq[t]}{d}{x}$ be a \sqf{} polynomial. Let $d > 0$ and let $\cv_1,\cv_2,\dots,\cv_d \in \fq[t]$. In this example we perturb two coefficients of each of the polynomials $\cv_1,\cv_2,\dots,\cv_d$ in order to obtain a \sqf{} value of $f$. In the case where $d=1$ this example is the same as Example~\ref{example single variable - two variable perturbation}. Let $\kappa_1,\kappa_2,\dots,\kappa_{d} \in \N$ such that $\kappa_i \not = 0 \mod p$ for any $\btwn{i}{1}{d}$. Then if $q$ is sufficiently large with respect to $\deg_{\vec{x}} f,\deg_t f$ and $\deg \cv_i,\kappa_i$ for $1 \leq i \leq d$, then there exist $\beta_1,\dots,\beta_{2d} \in \fq$ such that
\begin{equation}\label{equation simple perturbation generalization}
f(t,\beta_1+t^{\kappa_1}\beta_{d+1}+\cv_1,\beta_2+t^{\kappa_2}\beta_{d+2}+\cv_2,\dots,\beta_d+t^{\kappa_d}\beta_{2d}+\cv_d)
\end{equation}
is \sqf{}. In particular, if $\kappa_i = 1,\;\forall\btwn{i}{1}{d}$, a \sqf{} value of $f$ is obtained by perturbing the first two coefficients of $\cv_1,\dots,\cv_d$.
\end{example}

\section{The discriminant and constant assignments over a general field}\label{section the discriminant and constant assignments over a general field}
In this section we work over a general field $\FF$ which is not necessarily finite. Let $f \in \polring{\FF[t]}{d}{x}$. We first consider a special case of the main question we are concerned with, that of substituting $x_1,\dots,x_d$ with constants. By that we mean, we consider $f(t,\beta_1,\dots,\beta_d) \in \FF[t]$ where $\beta_1,\dots,\beta_d \in \FF$.

In the special case of constant assignments, the main question we are concerned with, is to infer from the assumption that $f$ is \sqf{}, that an assignment $f(t,\beta_1,\dots,\beta_d)$ is \sqf{}. Instead of drawing such a connection between $f$ being \sqf{} to $f(t,\beta_1,\dots,\beta_d)$ being \sqf{}, Lemma~\ref{lem main - squarefree values of polynomial separable in t}, which is the main lemma of this section, shows a connection between $f$ being separable in $t$ to $f(t,\beta_1,\dots,\beta_d)$ being separable in $t$. This connection is given by the existence of a polynomial $P \in \FF[x_1,\dots,x_d]$ which satisfies the following property: if $f$ is separable in $t$ then $P$ is not the zero polynomial, while if $f(t,\beta_1,\dots,\beta_d)$ is not separable in $t$ then $P(\beta_1,\dots,\beta_d)$ is zero. Since any non \sqf{} polynomial in $\FF[t]$ is in particular not separable, $P(\beta_1,\dots,\beta_d)=0$ for any $\beta_1,\dots,\beta_d$ such that $f(t,\beta_1,\dots,\beta_d)$ is not \sqf{}. Later this fact will be used in the proof of the main result.

In this section $|$ denotes assignment. Hence $$f|_{x_1=\beta_1,\dots,x_d=\beta_d}(t)=f(t,\beta_1,\dots,\beta_d) \in \FF[t].$$
Also, $\Delta$ denotes the discriminant of a polynomial as defined in Section~\ref{section Definitions and notations}.

The main observation, which is also the motivation for using the discriminant, is that $f|_{x_1=\beta_1,\dots,x_d=\beta_d}$ has a multiple root in $\overline{\FF}$ if and only if $\Delta (f|_{x_1=\beta_1,\dots,x_d=\beta_d})=0$. Hence, the constant assignments $\beta_1,\dots,\beta_d$ which result in a non separable polynomial $f(t,\beta_1,\dots,\beta_d) \in \FF[t]$, can be identified as those that make the discriminant vanish. The last fact can be used in order to define the polynomial $P$, as we now show.

Let $D^k \in \Z[x_0,\dots,x_k]$ be the polynomial which expresses the discriminant of a polynomial of degree $k$ in terms of its coefficients. Let $f \in \FF[x]$ be a polynomial such that $\deg f \leq k$, $f=\sum_{i=0}^{k} \delta_ix^i$. If $\delta_k \not = 0$, then the discriminant of $f$ in terms of its coefficients is given by $D^k$
$$
\Delta f = D^k(\delta_0,\dots,\delta_k).
$$
We use the notation $D^k f:=D^k(\delta_0,\dots,\delta_k)$. For example if $f(x)=\delta_2x^2+\delta_1x+\delta_0$, then $D^k f=\delta_1^2-4\delta_2 \delta_0$. We emphasize the distinction between $\Delta f$ and $D^k f$. If $\deg f=k$ then indeed $\Delta f = D^k f$. But if $\deg f < k$ then this is not necessarily true.

If $f \in \polring{\FF}{d}{x}$ is a multivariate polynomial then the notation $D^k_{x_i}$ will mean $D^{k} f$ where $f$ is viewed as a polynomial in variable $x_i$ over
$$\FF[x_1,\dots,x_{i-1},x_{i+1},\dots,x_d].$$
Namely, if
$$
f(x_1,\dots,x_d)=\sum_{j=0}^k \delta_j(x_1,\dots,x_{i-1},x_{i+1},\dots,x_d)x_i^j
$$
then
$$
D^k_{x_i} (f) = D^k(\delta_0(x_1,\dots,x_{i-1},x_{i+1},\dots,x_d),\dots,\delta_k(x_1,\dots,x_{i-1},x_{i+1},\dots,x_d)).$$
In this case $D^k_{x_i} (f) \in \FF[x_1,\dots,x_{i-1},x_{i+1},\dots,x_d]$.

The following fact, which we state as a lemma, is a direct consequence of the fact that $D^k$ is a polynomial in $\Z[x_0,\dots,x_k]$ which depends only on $k$ and is the same regardless of the base field.

\begin{lem}\label{lem discriminant calculation}
Let $f \in \FF[t][x_1,\dots,x_d]$ where $\deg_t f \leq k$. Let $\beta_1,\dots,\beta_d \in \FF$. Then
\begin{equation}\label{equation switch order of discriminant and assignment}
(D^k_t f) |_{x_1=\beta_1,\dots,x_d=\beta_d}=D^k (f|_{x_1=\beta_1,\dots,x_d=\beta_d}).
\end{equation}
\end{lem}

The right hand side of Eq.~\ref{equation switch order of discriminant and assignment} means first assigning $x_1=\beta_1,\dots,x_d=\beta_d$ to the polynomial $f$. The result is a polynomial in $\FF[t]$. Then $D^k$ is applied to the result. The left hand side of ~\ref{equation switch order of discriminant and assignment} means first applying $D_t^k f$. The result is a polynomial in $\FF[x_1,\dots,x_d]$. Then assigning $x_1=\beta_1,\dots,x_d=\beta_d$ to the result. The order of operations is opposite in the two expressions. The lemma asserts that the two are equal. Let $\mathcal{F}_k :=\{f \in \polring{\FF[t]}{d}{x}: \deg_t f \leq k\}$. Then the statement of the lemma is summarized by the commutative diagram below which holds for any polynomial $f \in \mathcal{F}_k$.

$$
\xymatrix{
\polring{\FF}{d}{x} \ar[d]^{|_{x_1=\beta_1,\dots,x_d=\beta_d}}&
\mathcal{F}_k \ar[l]^{D^k_t} \ar[d]^{|_{x_1=\beta_1,\dots,x_d=\beta_d}}\\
\FF &\FF[t] \cap \mathcal{F}_k\ar[l]^{D^k}}
$$

\begin{proof}

$f(t,x_1,\dots,x_d)=\sum_{i=0}^{k}\delta_i(x_1,\dots,x_d)t^i$. Hence

\begin{equation}\label{equation descriminant in x}
D^k_t f = D^k (\delta_0(x_1,\dots,x_d),\delta_1(x_1,\dots,x_d),\dots,\delta_k(x_1,\dots,x_d)).
\end{equation}

Hence
\begin{equation}\label{equation opposite operation first discriminant}
(D^k_t f)|_{x_1=\beta_1,\dots,x_d=\beta_d}=D^k(\delta_0(\beta_1,\dots,\beta_d),\delta_1(\beta_1,\dots,\beta_d),\dots,\delta_k(\beta_1,\dots,\beta_d)).
\end{equation}

Now, $f|_{x_1=\beta_1,\dots,x_d=\beta_d}=f(t,\beta_1,\dots,\beta_d)=\sum_{i=0}^{k}\delta_i(\beta_1,\dots,\beta_d)t^i$. Hence

\begin{equation}\label{equation opposite operation first assigmment}
D^k (f|_{x_1=\beta_1,\dots,x_k=\beta_k}) = D^k (\delta_0(\beta_1,\dots,\beta_k),\delta_1(\beta_1,\dots,\beta_d),\dots,\delta_k(\beta_1,\dots,\beta_d)).
\end{equation}

As we mentioned before the proof, $D^k$ in Eq.~\ref{equation opposite operation first discriminant} and Eq.~\ref{equation opposite operation first assigmment} is the same polynomial, although the base field is different, hence the expressions are equal.
\end{proof}

\begin{lem} \label{lem main - squarefree values of polynomial separable in t}
Let $\FF$ be a field. Let $f \in \polring{\FF[t]}{d}{x}$ be a polynomial. Then there exists a polynomial $P \in \FF[x_1,\dots,x_d]$ such that
\begin{multline}
\{(\beta_1,\dots,\beta_d) \in \FF^d : f(t,\beta_1,\dots,\beta_d) \mbox{ is not separable} \} \subseteq\\ \{(\beta_1,\dots,\beta_d) \in \FF^d:P(\beta_1,\dots,\beta_d)=0 \}
\end{multline}
where
\begin{equation}\label{equation single variable main lemma pol bound}
\deg P \leq (2\deg_t f -1)\deg_{\vec{x}} f.
\end{equation}

The polynomial $P$ is non-zero if and only if $f$ is separable in $t$.
\end{lem}

\begin{proof}
Let $k=\deg_t f$. Write $f=\sum_{i=0}^{k}\delta_i(x_1,\dots,x_d)t^i$ where $\delta_k \not = 0$. Let $P \in \polring{\FF}{d}{x}$ defined by $P:=(D^k_t f)\cdot \delta_k$. Note that $\delta_k$ is nonzero, and $D^k_t f = \Delta_t f$ is nonzero if and only if $f$ is separable in $t$. Hence $P$ is nonzero if and only if $f$ is separable in $t$.

Now suppose $\beta_1,\dots,\beta_d \in \FF$ are such that $f(t,\beta_1,\dots,\beta_d)$ is not separable. We need to show that $P(\beta_1,\dots,\beta_d)=0$.
Assume first $\deg(f|_{x_1=\beta_1,\dots,x_d=\beta_d}) = k$. Then $\Delta (f|_{x_1=\beta_1,\dots,x_d=\beta_d}) = D^k (f|_{x_1=\beta_1,\dots,x_d=\beta_d})$. Since $f|_{x_1=\beta_1,\dots,x_d=\beta_d}$ is not separable, $\Delta ( f|_{x_1=\beta_1,\dots,x_d=\beta_d}) = 0$. By Lemma~\ref{lem discriminant calculation} we get:
$$
0=\Delta (f|_{x_1=\beta_1,\dots,x_d=\beta_d})=D^k (f|_{x_1=\beta_1,\dots,x_d=\beta_d}) = (D^k_t f)|_{x_1=\beta_1,\dots,x_d=\beta_d}.
$$
Hence $(D^k_t f)(\beta_1,\dots,\beta_d) = 0$.
Now assume $\deg(f|_{x_1=\beta_1,\dots,x_d=\beta_d}) < k$ then $\delta_k(\beta_1,\dots,\beta_d)=0$. Hence in any case $P(\beta_1,\dots,\beta_d)=0$.

It remains to bound the degree of $P$.
$$
D^k_t f= D^k(\delta_0(x_1,\dots,x_d),\dots,\delta_k(x_1,\dots,x_d))
$$

$$
P = \delta_k D^k_t f = \delta_k (x_1,\dots,x_d) D^k(\delta_0(x_1,\dots,x_d),\dots,\delta_k(x_1,\dots,x_d))
$$

$D^k \in \Z[x_0,\dots,x_k]$ is a homogenous polynomial of total degree $2k-2$.
$\deg \delta_i \leq \deg_{\vec{x}} f$ for any $i$, $0 \leq i \leq k$. We get
$$
\deg (\delta_k D^k_t) \leq \deg_{\vec{x}} f+ (2k-2)\deg_{\vec{x}} f = (2k-1)\deg_{\vec{x}} f = (2\deg_t f -1)\deg_{\vec{x}} f.
$$
\end{proof}

\section{Background and general facts that are used to prove the main results}\label{section Preliminary facts and results}

\subsection{Separable and \sqf{} polynomials}

Since the following theorem and two consequences of it are not related to the main subject of the note, we state them in this section without a proof. As we did not find the exact theorems in another source, for the completeness of the note we prove them at the appendix.


\begin{thm}\label{thm absolutely square-free equivalent conditions}Let $\FF$ be a field of positive characteristic $p$, and let $\overline{\FF}$ be the algebraic closure of $\FF$. Let
\begin{equation}\label{equation ajoin roots of order p}
\FF^{\frac{1}{p}}=\FF\left (\left \{c^{\frac{1}{p}}: c \in \FF\right \}\right ).
\end{equation}
Let $f \in \polring{\FF}{d}{x}$. The following are equivalent:
\begin{enumerate}
  \item \label{absolutely squarefree 1} $f$ is square-free as an element in $\polring{\overline{\FF}}{d}{x}$.
  \item \label{absolutely squarefree 2}$f$ is square-free as an element in $\polring{\FF^{\frac{1}{p}}}{d}{x}$.
  \item \label{absolutely squarefree 3} $f$ is square-free as an element in $\polring{\FF}{d}{x}$, and $f$ does not have an irreducible factor $g$ such that $g \in \polring{\FF}{d}{x^p}$.
\end{enumerate}
\end{thm}

\begin{cor}\label{cor square-free equals absolutely square-free}
Let $\FF$ be a perfect field. Let $\overline{\FF}$ be the algebraic closure of $\FF$. Let $f \in \polring{\FF}{d}{x}$. Then $f$ is \sqf{} in $\polring{\FF}{d}{x}$ if and only if $f$ is \sqf{} in $\polring{\overline{\FF}}{d}{x}$.
\end{cor}

\begin{cor}\label{cor if f is separable in t it has a factor in ftp}
Let $\FF$ be a field. Let $f \in \polring{\FF[t]}{d}{x}$ be a \sqf{} polynomial.
\begin{enumerate}
\item\label{part if char=0 then f is separable}If $\chr(\FF)=0$ then $f$ is separable in $t$.
\item\label{part part if char>0 then f has a factor in ftp} If $\chr(\FF)>0$ and $f$ is not separable in $t$ then there exists an irreducible $g \in \polring{\FF[t^p]}{d}{x}$ which divides $f$. Also $\deg_t g > 0$.
\end{enumerate}
\end{cor}

\subsection{Properties of homomorphism}\label{subsubsection properties of homomorphisms}

The following two lemmas can be viewed as consequences of the structure preserving nature of homomorphisms.

\begin{lem}\label{lem homomorphism preserves invertible elements}
Let $D$ be an integral domain. Let $\Psi: D \rightarrow D$ be a homomorphism. If $u \in D$ is invertible then $\Psi(u)$ is invertible.
\end{lem}
\begin{proof}
Since $u$ is invertible, there exists $u^{-1} \in D$ an inverse of $u$, and $1=\Psi(1)=\Psi(u)\Psi (u^{-1})$.
\end{proof}

\begin{lem}\label{lem general automorphism preserves irreducibility}
Let $D$ be an integral domain. Let $\Psi:D \rightarrow D$ be an automorphism. Let $r \in D$. Then $r$ is irreducible if and only if $\Psi(r)$ is irreducible.
\end{lem}

\begin{proof}
We show that if $\Psi(r)$ is irreducible then $r$ is irreducible. The opposite direction follows by symmetry when using the identity $r=\Psi^{-1}(\Psi(r))$. Suppose $r=r_1r_2$ where $r_1,r_2 \in D$. Then $\Psi(r)=\Psi(r_1r_2)=\Psi(r_1)\Psi(r_2)$. Since $\Psi(r)$ is irreducible one of $\Psi(r_1),\Psi(r_2)$ must be invertible. Suppose without loss of generality that $\Psi(r_2)$ is invertible. Hence by Lemma ~\ref{lem homomorphism preserves invertible elements} $r_2 = \Psi^{-1}(\Psi(r_2))$ is invertible. Since $r_1,r_2$ is an arbitrary factorization of $r$ it follows that $r$ is irreducible.
\end{proof}

\subsection{Derivation of rational functions}\label{subsubsection derivation rings}
Let $R$ be a ring. An operation $\delta : R \rightarrow R$
is called \textbf{derivation operator} if it satisfies the following two requirements for any two elements $a,b \in R$:
\begin{enumerate}
\item$\delta(a+b)=\delta(a)+\delta(b)$.
\item $\delta(ab)=\delta (a)b+a\delta(b)$.
\end{enumerate}

For example, $\frac{\partial }{\partial x_i}$ is a derivation operator of $\polring{R}{d}{x}$, as it satisfies both properties of a derivation.

Let $R$ be a ring and let $S$ be a multiplicative subset of $R$. That is, $S$ is such that for any $s_1,s_2 \in S$, $s_1s_2 \in S$. By a standard construction there exists a ring which contains quotients $\frac{a}{s}$ where $a \in R$ and $s \in S$, which we denote by $S^{-1}R$. Any derivation operator of $R$ can be extended to a derivation operator of $S^{-1}R$, where the derivation in $S^{-1}R$ is given by the usual quotient rule for derivatives. We state this fact in the following proposition.

\begin{prop}\label{prop extension of a derivation ring}
Let $R$ be a ring, let $\delta$ be a derivation operator of $R$, and let $S$ be a multiplicative subset of $R$. Let $r_1,r_2 \in R, s_1,s_2 \in S$ such that $\frac{r_1}{s_1}=\frac{r_2}{s_2}$. Then
$$
\frac{\delta(r_1)s_1-r_1\delta(s_1)}{s_1^2}=\frac{\delta(r_2)s_2-r_2\delta(s_2)}{s_2^2}.
$$
\end{prop}

A proof of the above proposition and more information about the extension of the derivation operator to $S^{-1}R$ can be found in \cite{Kolchin} Chapter 1.

We are interested in the case where the ring is a polynomial ring over a field, $\polring{\FF}{d}{x}$, and the derivation operators are given by the formal partial derivatives $\partiald{f}{x_1},\dots,\partiald{f}{x_d}$. We will also derive rational functions over a field, and the meaning of that is made precise by the general facts about derivations ring which are described above.

For convenience of reference, we state the following known fact about derivations which is used in this note.


\begin{prop}[Chain Rule]
Let $f \in \polfield{\FF}{d}{x}$, let $g=(g_1,\dots,g_d)$ where $g_1,\dots,g_d \in \FF(x)$, and let $f \circ g \in \FF(x)$, $f \circ g(x)=f(g_1(x),\dots,g_d(x))$ then
\begin{equation}\label{equation the chain rule}
\ordinaryd{f \circ g}{x} = \sum_{i=1}^d \partiald{f}{x_i}\ordinaryd{g_i}{x}.
\end{equation}
\end{prop}

\section{Proof of Theorem~\ref{thm perturbations of a single variable polynomial finit field}}\label{section Proof of Theorem single variable}

The main steps in the proof of Theorem~\ref{thm perturbations of a single variable polynomial finit field} are as follows. First, we introduce Theorem~\ref{thm single variable - algebraic}, which is an algebraic theorem which holds for any field $\FF$. Secondly, we show that Theorem~\ref{thm perturbations of a single variable polynomial finit field} is a consequence of Theorem~\ref{thm single variable - algebraic}, when the latter is applied in the special case where $\FF$ is assumed to be a finite filed. Finally, we prove Theorem~\ref{thm single variable - algebraic}. For that purpose, we apply Lemma~\ref{lem main - squarefree values of polynomial separable in t} which was introduced in Section~\ref{section the discriminant and constant assignments over a general field}.

\subsection{Reduction to an algebraic theorem which holds for any field}\label{subsection single reducton to an algebraic theorem}

The following lemma provides an elementary upper bound on the number of zeros of a multivariate polynomial over a finite field. A proof can be found in \cite{Schmidt} Chapter 4.

\begin{lem}\label{lem elementary bound on number of zeros}
Let $P \in \polring{\fq}{d}{x}$ be a non-zero polynomial of total degree $n$. Then the number $\mathcal{N}_P$ of zeros of $P(x_1,\dots,x_d)$ in $\fq^d$ satisfies
\begin{equation}\label{equation elementary bound on number of zeros}
\mathcal{N}_P \leq nq^{d-1}.
\end{equation}
\end{lem}

Theorem~\ref{thm perturbations of a single variable polynomial finit field} is a consequence of the following theorem which holds for any field $\FF$.

\begin{thm}\label{thm single variable - algebraic}
Let $\FF$ be a field, and let $\overline{\FF}$ be an algebraic closure of $\FF$. Let $f \in \FF[t][x]$ be a polynomial which is \sqf{} in $\overline{\FF}[t][x]$. Let $\av,\bv,\cv \in \FF[t]$ such that $\gcd(\av,\bv)=1$. Let $N \in \N$. Assume $\deg_x f,\deg_t f, \|(\av,\bv,\cv)\| \leq N$. Then there exists a polynomial $P_{f,\av,\bv,\cv} \in \FF[x_1,x_2]$ which depends on $\av,\bv,\cv$ and $f$ such that
\begin{multline}\label{equation polynomial P property}
\{(\beta_1,\beta_2) \in \FF^2 : f(t,\av(t)\beta_1+\bv(t)\beta_2+\cv(t)) \mbox{ is not separable}\}\\ \subseteq \{(\beta_1,\beta_2) \in \FF^2 : P_{f,\av,\bv,\cv}(\beta_1,\beta_2)=0\}.
\end{multline}
Moreover, there exists a constant $\tilde{C}(N)$ which depends only on $N$ such that$$\deg P_{f,\av,\bv,\cv} \leq \tilde{C}(N).$$
$P_{f,\av,\bv,\cv}$ is non-zero if at least one of the following holds:
\begin{enumerate}
\item \label{part thm sigle variable - albgebraic - single variable perturbation}
$p=0$ or $p>C(N)$ where $C(N)$ is a constant which depends only on $N$.
\item\label{part thm sigle variable - albgebraic - two variables perturbation}
$\frac{\bv}{\av} \not \in \FF(t^p)$ where $\av \not = 0$.
\end{enumerate}
\end{thm}

We now show why Theorem~\ref{thm perturbations of a single variable polynomial finit field} follows from Theorem~\ref{thm single variable - algebraic} in the case where $\FF$ is a finite field. First, over a finite field the requirement in Theorem~\ref{thm single variable - algebraic} that $f$ is \sqf{} in $\overline{\FF}[t][x]$ can be replaced by the requirement that $f$ is \sqf{}. Over a finite field the two requirements are equivalent by Corollary~\ref{cor square-free equals absolutely square-free}, since $\fq$ is a perfect field. Likewise, by the same corollary over a finite field $f(t,\av(t)\beta_1+\bv(t)\beta_2+\cv(t))$ being separable and $f(t,\av(t)\beta_1+\bv(t)\beta_2+\cv(t))$ being \sqf{} can be used interchangeably, since the two are equivalent over a perfect field.

With the assumptions and definitions as in Theorem~\ref{thm perturbations of a single variable polynomial finit field}, we now show why the estimate in Eq.~\ref{equation main theorem single variable estimate} of Theorem~\ref{thm perturbations of a single variable polynomial finit field} follows. Let $\squarefsing{\fq}{f}^c$ be the complement of $\squarefsing{\fq}{f}$ in $\fq[t]$. The error $E(q)$ of estimating
$$\frac{\#(\squarefsing{\fq}{f} \bigcap \persing{\fq}{\av}{\bv}{\cv})}{\#\persing{\fq}{\av}{\bv}{\cv}}$$
by $1$ is given by
\begin{multline}
E(q) = 1-\frac{\#(\squarefsing{\fq}{f} \bigcap \persing{\fq}{\av}{\bv}{\cv})}{\#\persing{\fq}{\av}{\bv}{\cv}} = \frac{\#(\squarefsing{\fq}{f}^c \bigcap \persing{\fq}{\av}{\bv}{\cv})}{\#\persing{\fq}{\av}{\bv}{\cv}}=\\ \frac{\#\{(\beta_1,\beta_2) \in \fq^2 : f(t,\av(t)\beta_1+\bv(t)\beta_2+\cv(t)) \mbox{ is not \sqf{}}\}}{q^2} \leq\\
\frac{\#\{(\beta_1,\beta_2) \in \fq^2 : P_{f,\av,\bv,\cv}(\beta_1,\beta_2)=0\}}{q^2}.
\end{multline}

Assume first that $\av,\bv,\cv$ are such that $P_{f,\av,\bv,\cv}$ is nonzero. Then applying Lemma~\ref{lem elementary bound on number of zeros}
$$
E(q) \leq \frac{\tilde{C}(N)q}{q^2}.
$$
Hence keeping $N$ fixed while $q \to \infty$, $E(q)=O \left (\frac{1}{q} \right )$.

It remains to show why it follows from the assumptions of Theorem~\ref{thm perturbations of a single variable polynomial finit field} that $P_{f,\av,\bv,\cv}$ is non-zero. According to Theorem~\ref{thm single variable - algebraic} $P_{f,\av,\bv,\cv}$ is non-zero if at least one of (\ref{part thm sigle variable - albgebraic - single variable perturbation}) and (\ref{part thm sigle variable - albgebraic - two variables perturbation}) in the same theorem holds. Indeed, by letting $C(N)$ in Theorem~\ref{thm perturbations of a single variable polynomial finit field} be the same as $C(N)$ in Theorem~\ref{thm single variable - algebraic}, (\ref{part thm sigle variable - albgebraic - single variable perturbation}) and (\ref{part thm sigle variable - albgebraic - two variables perturbation}) in Theorem~\ref{thm single variable - algebraic} are the same as (\ref{part single variable theorem large char}) and (\ref{part single variable theorem arbitrary char}) in Theorem~\ref{thm perturbations of a single variable polynomial finit field}. Hence by the assumptions of Theorem~\ref{thm perturbations of a single variable polynomial finit field} at least one of the two holds.

\begin{remark}
The statement that $P_{f,\av,\bv,\cv}$ is nonzero is crucial for the reduction, since this is required in order for the bound in Eq.~\ref{equation elementary bound on number of zeros} to hold. In the proof of Theorem~\ref{thm single variable - algebraic} the existence of $P_{f,\av,\bv,\cv}$ will be provided by Lemma~\ref{lem main - squarefree values of polynomial separable in t} which we proved in Section~\ref{section the discriminant and constant assignments over a general field}, and the part of Theorem~\ref{thm single variable - algebraic} about $P_{f,\av,\bv,\cv}$ not being the zero polynomial will be given by the same lemma.
\end{remark}

\subsection{A single coefficient perturbation}\label{section A single variable perturbation}We now turn to prove Theorem ~\ref{thm single variable - algebraic}. To illustrate the main steps in the proof, we start by proving a special case of Theorem~\ref{thm single variable - algebraic} where $\chr (\FF)$ is large, namely (\ref{part thm sigle variable - albgebraic - single variable perturbation}) in Theorem~\ref{thm single variable - algebraic} holds, $\cv \in \FF[t]$ is an arbitrary polynomial, $\av=1$ and $\bv=0$, which corresponds to perturbations of the free coefficient of $\cv$. In the case where $\FF$ is a finite field, by the reduction of Section~\ref{subsection single reducton to an algebraic theorem},
this implies the special case of Theorem~\ref{thm perturbations of a single variable polynomial finit field} which is given in Example~\ref{example single variable - one variable perturbation} of Section~\ref{section the main results}.

\begin{defn}
Let $R,R_1,R_2$ be rings such that $R \subseteq R_1,R \subseteq R_2$. A homomorphism $\Psi: R_1 \rightarrow R_2$ is a \textbf{$R$- homomorphism} if for any $r \in R, \Psi(r)=r$.
\end{defn}

Let $\cv \in \FF[t]$. There exist a unique $\FF[t]$-automorphism $\Psi: \FF[t][x] \rightarrow \FF[t][x]$ such that $x \mapsto x+\cv$. This automorphism is given by $f(t,x) \mapsto f(t,x+\cv(t))$.

\begin{lem}\label{lem single variable large char - rho(f) is separable}
Let $f \in \FF[t][x]$ be a \sqf{} polynomial, and let $\cv \in \FF[t]$. Let $\Psi : \FF[t][x] \rightarrow \FF[t][x]$ be the unique $\FF[t]$-automorphism defined by $x \mapsto x+\cv$. If $p=0$ or $p>\deg_x (f) \deg (\cv)+\deg_t (f)$ then $\Psi(f)$ is separable in $t$.
\end{lem}

\begin{proof}
Let $f=\prod_{i=1}^k f_i$ be a factorization of $f$ into irreducible factors. Then by homomorphism properties

\begin{equation} \label{equation single varialble - single var perturb}
\Psi(f) = \prod_{i=1}^k \Psi(f_i).
\end{equation}

Since $\Psi$ is an automorphism, by Lemma~\ref{lem general automorphism preserves irreducibility} $\Psi(f_1),\dots,\Psi(f_k)$ are irreducible. Hence Eq.~\ref{equation single varialble - single var perturb} provides a factorization of $\Psi(f)$ into irreducible factors in $\FF[t][x]$. We now show that $\Psi(f)$ is \sqf{}. Suppose there exist $f_i,f_j$ where $i \not = j$ such that $\Psi(f_i)$ and $\Psi(f_j)$ are associated, that is $\Psi(f_i) = \alpha\Psi(f_j)$ where $\alpha \in \FF$. Since $\Psi(\alpha)=\alpha$ we have $\Psi(f_i)=\Psi(f_j)\Psi(\alpha)=\Psi(f_j\alpha)$. But $\Psi$ is injective, hence $f_i=f_j\alpha$. But that is a contradiction to the assumption that $f$ is \sqf{}. Hence $\Psi(f)$ is \sqf{}. If $p=0$ then we are done, since it follows by Corollary~\ref{cor if f is separable in t it has a factor in ftp} part~\ref{part if char=0 then f is separable} that $\Psi(f)$ is separable in $t$.

Now suppose $p>0$, and suppose on the contrary that $\Psi(f)$ is not separable as polynomial in $t$. By Corollary~\ref{cor if f is separable in t it has a factor in ftp} part~\ref{part part if char>0 then f has a factor in ftp} $\Psi(f)$ has an irreducible factor in $\FF[t^p][x]$, where its degree in $t$ is not zero. Without loss of generality, we can assume this irreducible factor is $\Psi(f_i)$, for some $i$, $1 \leq i \leq k$. But the latter cannot hold since for any $\btwn{i}{1}{k}$
$$
\deg_t (\Psi(f_i)) \leq \deg_t (\Psi(f)) \leq \deg (\cv) \deg_x (f)+\deg_t (f) < p,
$$
where the last inequality is by our assumption on $p$.
\end{proof}

We now prove Theorem~\ref{thm single variable - algebraic} for the special case of a single variable perturbation, where we assume (\ref{part thm sigle variable - albgebraic - single variable perturbation}) of Theorem~\ref{thm single variable - algebraic} holds.

\begin{proof}Let $\Psi: \FF[t][x] \rightarrow \FF[t][x]$ be the unique $\FF[t]$- automorphism which is given by $x \mapsto \cv$. Let $\tilde{f}:=\Psi(f)$. The proof will follow by applying Lemma~\ref{lem main - squarefree values of polynomial separable in t} to $\tilde{f}$. Let $C(N)$ be sufficiently large such that $C(N) \geq \deg(c)\deg_{x}(f)+\deg_t(f)$ holds. Since we assume that (\ref{part thm sigle variable - albgebraic - single variable perturbation}) in Theorem~\ref{thm single variable - algebraic} holds, by Lemma~\ref{lem single variable large char - rho(f) is separable}, $\tilde{f}$ is separable as polynomial in variable $t$. Hence by Lemma~\ref{lem main - squarefree values of polynomial separable in t} there exists a non zero $P \in \FF[x]$ such that
$$
\{\beta \in \FF: \tilde{f}(t,\beta) \mbox{ is not separable}\} \subseteq \{\beta \in \FF: P(\beta)=0\}.
$$
But
$$
\{\beta \in \FF: \tilde{f}(t,\beta) \mbox{ is not separable}\}=\{\beta \in \FF: f(t,\beta+\cv(t)) \mbox{ is not separable}\}
$$
Also,
$$
\deg_t \tilde{f} \leq \deg_x f \deg \cv + \deg_t f
$$
and
$$
\deg_x \tilde{f} = \deg_x f.
$$
Assigning this into the bound which is given by Eq.~\ref{equation single variable main lemma pol bound} in Lemma~\ref{lem main - squarefree values of polynomial separable in t} we get
$$
\deg P \leq (2(\deg_x f \deg \cv + \deg_t f)-1)\deg_x f.
$$
Let $\tilde{C}(N)$ be such that $\tilde{C}(N) \geq (2(\deg_x f\deg c + \deg_t f)-1)\deg_x f$, then the bound on $\deg P$ holds as required.
\end{proof}

\subsection{Proof of Theorem~\ref{thm single variable - algebraic}}\label{subsection Proof of Theorem thm single variable - algebraic}
The proof of Theorem~\ref{thm single variable - algebraic} follows similar lines as the proof of a special case of the same theorem in the previous section, only that now we consider more homomorphisms of $\FF[t][x]$ other than the one which maps $x$ to $x+\cv$ where $\cv \in \FF[t]$. Also, in Lemma~\ref{lem single variable large char - rho(f) is separable} and its proof we assumed $p$ is sufficiently large. Lemma~\ref{lem single variable low char - image of f is separable in t} which we prove in this section generalizes Lemma~\ref{lem single variable large char - rho(f) is separable} and its proof to the case of an arbitrary characteristic $p$.

Let $R$ be a ring. Every $R$- homomorphism $R[x_1] \rightarrow R[x_1,x_2]$ is uniquely determined by the image of $x_1$. Conversely, for any $f \in R[x_1,x_2]$ there is an $R$- homomorphism $R[x_1] \rightarrow R[x_1,x_2]$ such that $x_1 \mapsto f$. This homomorphism is given by $g \mapsto g \circ f$ for any $g \in R[x_1]$. We denote by $\Psi_f$ the unique $R$-homomorphism $\Psi_f: R[x_1] \rightarrow R[x_1,x_2]$ such that $x_1 \mapsto f$.

In the following lemma we assume $D$ is an integral domain. For our purposes, we need only the case where $D=\FF[t]$.
\begin{lem}\label{lem single variable large char - f is irr then so is rho(f)}
Let $D$ be an integral domain. Let $\av,\bv,\cv \in D$, such that $\av \not = 0$ and $l \in D[x_1,x_2],l(x_1,x_2):=\av x_1+\bv x_2+\cv$. Let $\Psi_l: D[x_1] \rightarrow D[x_1,x_2]$ be a $D$- homomorphism defined by $x_1 \mapsto l$.
\begin{enumerate}
\item \label{single variable large char - can be extended} Let $K$ be the field of fractions of $D$. $\Psi_l$ can be extended to an automorphism $\tilde{\Psi}: K[x_1,x_2] \rightarrow K[x_1,x_2]$. In particular, $\Psi_l$ is injective.
\item \label{single variable large char - is irreducible in kx} If $D$ is a field, and if $f$ is irreducible in $D[x_1]$ then $\Psi_l(f)$ is irreducible in $D[x_1,x_2]$.
\item \label{single variable large char - is primitive} If $D$ is a unique factorization domain, $\gcd(\av,\bv)=1$ and $f$ is primitive in $D[x_1]$ then the greatest common divisor of the coefficients of $\Psi_l(f)$ in $D$ is $1$.
\item \label{single variable large char - is irreducible} If $D$ is a unique factorization domain, $\gcd(\av,\bv)=1$ and $f$ is irreducible in $D[x_1]$ then $\Psi_l(f)$ is irreducible in $D[x_1,x_2]$.
\end{enumerate}
\end{lem}

\begin{proof}
\ref{single variable large char - can be extended}: $\Psi_{l}$ can be extended to a $K$- homomorphism $\tilde{\Psi}: K[x_1,x_2] \rightarrow K[x_1,x_2]$ by the following rule
$$x_1 \mapsto \av x_1+\bv x_2+\cv,$$
$$x_2 \mapsto x_2.$$
$\tilde{\Psi}$ is an automorphism, with an inverse homomorphism $\tilde{\Psi}^{-1}: K[x_1,x_2] \rightarrow K[x_1,x_2]$ which is given by
$$
x_1 \mapsto \frac{1}{\av}(x_1-\bv x_2-\cv)
$$
$$
x_2 \mapsto x_2.
$$

Hence the claim follows.

\ref{single variable large char - is irreducible in kx}: It follows from the assumption that $f$ is irreducible in $D[x_1]$, that $f$ is irreducible also  in $D[x_1,x_2]$. This is because for every factorization of $f$, $f=g_1g_2$ where $g_1,g_2 \in D[x_1,x_2]$, $\deg_{x_2}g_1=\deg_{x_2}g_2=\deg_{x_2}f=0$. Hence $g_1,g_2 \in D[x_1]$. By part~\ref{single variable large char - can be extended} of the lemma, since $D=K$ in this case, $\Psi_l$ can be extended to an automorphism $\tilde{\Psi}$ of $D[x_1,x_2]$. Since we assume $f$ is irreducible in $D[x_1,x_2]$, by Lemma~\ref{lem general automorphism preserves irreducibility} $\tilde{\Psi}(f)$ is irreducible in $D[x_1,x_2]$. But $\Psi_l(f)=\tilde{\Psi}(f)$. Hence $\Psi_l(f)$ is irreducible in $D[x_1,x_2]$.

\ref{single variable large char - is primitive}: Let $h$ be a prime element in $D$. Then $\langle h \rangle$ is a prime ideal. Hence $D/\langle h \rangle$ is an integral domain. Let $\overline{D}=D/\langle h \rangle$.

We now define few notations. For an element $c \in D$ denote by $\overline{c}$ the equivalence class of $c$ in $\overline{D}$. For a polynomial $g \in \FF[x_1]$, $g=\sum_{i=0}^{n}c_ix_1^i$ denote $\overline{g}=\sum_{i=0}^{n}\overline{c_i}x_1^i$. Likewise, for a polynomial $g \in \FF[x_1,x_2]$, $g=\sum_{0 \leq i\leq n_1,0 \leq j \leq n_2}c_{i,j}x_1^ix_2^j$ denote $\overline{g}=\sum_{0 \leq i \leq n_1,0 \leq j \leq n_2}\overline{c_{i,j}}x_1^ix_2^j$.

Let $\Psi_{\overline{l}}: \overline{D}[x_1] \rightarrow \overline{D}[x_1,x_2]$ be a $\overline{D}$- homomorphism defined by $x \mapsto \bar{l}$. Since $\gcd(\av,\bv)=1$, at least one of $\overline{\av},\overline{\bv}$ is nonzero. Hence by part~\ref{single variable large char - can be extended} $\overline{D}$ can be extended to an automorphism $\overline{K}[x_1,x_2] \rightarrow \overline{K}[x_1,x_2]$ where $\overline{K}$ is the field of fractions of $\overline{D}$. In particular, the kernel of $\Psi_{\overline{l}}$ is trivial. Since $f$ is primitive, hence $\overline{f} \not = 0 \mod \langle h \rangle$, it follows that $\Psi_{\overline{l}}(\bar{f}) \not = 0 \mod \langle h \rangle$. But $\overline{\Psi_l(f)}=\Psi_{\overline{l}}(\overline{f})$ by homomorphism properties. Hence $h$ does not divide all coefficients of $\Psi_l (f)$. Since $h$ is arbitrary the claim follows.

\ref{single variable large char - is irreducible}: Let $K$ be the field of fractions of $D$. Since $f$ is irreducible in $D[x_1]$, by Gauss's lemma for polynomials it is irreducible in $K[x_1]$ and primitive. By part~\ref{single variable large char - is irreducible in kx} applied on $f$ as an element in $K[x_1]$, $\Psi_l(f)$ is irreducible in $K[x_1,x_2]$. Since $f$ is primitive, by part~\ref{single variable large char - is primitive} the coefficients of $\Psi_l(f)$ in $D$ do not have a nontrivial common divisor. Since $\Psi_l(f)$ is irreducible in $K[x_1,x_2]$ and does not have a nontrivial factor in $D$, it is irreducible in $D[x_1,x_2]$.
\end{proof}

Although in the main theorem $\av,\bv,\cv$ and $f$ are assumed to be polynomials, for the following lemma we assume $\av,\bv,\cv$ and $f$ are rational functions and not necessarily polynomials, i.e. $\av,\bv,\cv \in \FF(t)$ and $f \in \FF(t,x_1)$. We will also perform formal derivations of functions in $\FF(t,x_1,x_2)$. The meaning of that is described in Section~\ref{subsubsection derivation rings}.

\begin{lem}\label{lem single variable low char - separable displacments}
Let $\FF$ be a field of positive characteristic $p$. Let $f \in \FF(t,x_1)$. Let $\av,\bv,\cv \in \FF(t)$ such that $\av \not = 0$ and $\frac{\bv}{\av} \not \in \FF(t^p)$. Suppose $f(t,\av(t)x_1+\bv(t)x_2+\cv(t)) \in \FF(t^p,x_1,x_2)$. Then $f \in \FF(t^p,x_1^p)$.
\end{lem}


In the proof of Lemma~\ref{lem single variable low char - separable displacments} we use the fact that $f \in \FF(t^p)$ if and only if $\ordinaryd{f}{t} = 0$, which is a known property of derivations.

\begin{proof}
First we note that it is sufficient to prove the lemma in the case where $\av=1$ and $\cv=0$. This is because if we define $\overline{f}(t,x_1):=f(t,\av(t) x_1+\cv(t))$, and define $\overline{\av}:=1,\overline{\bv}:=\frac{\bv}{\av},\overline{\cv}:=0$, then $f(t,\av(t)x_1+\bv(t)x_2+\cv(t))=\overline{f}(t,x_1+\frac{\bv}{\av}(t)x_2)
=\overline{f}(t,\overline{\av}(t)x_1+\overline{\bv}(t)x_2+\overline{\cv}(t))$. Proving the lemma for the special case will show that if $\overline{f}(t,x_1+\frac{\bv}{\av}(t)x_2) \in \FF(t^p,x_1^p)$ then $\overline{f} \in \FF(t^p,x_1^p)$, i.e. $\overline{f}=g(t^p,x_1^p)$ for some $g \in \FF(t,x_1)$. But then
\begin{multline}f\left (t,x_1\right )=\overline{f}\left (t,\frac{x_1-\cv(t)}{\av(t)}\right )=\\g\left (t^p, \left  (\frac{x_1-\cv(t)}{\av(t)}\right )^p\right )=g\left (t^p, \frac{x_1^p-\cv(t)^p}{\av(t)^p}\right ) \in \FF(t^p,x_1^p).
\end{multline}
Hence it is sufficient to prove the case where $f=\overline{f},\av=\overline{\av},\bv=\overline{\bv},\cv=\overline{\cv}$. From now on we assume $\av=1$ and $\cv=0$.

Now, let $\tilde{f}:=f(t,x_1+\bv(t) x_2)$. Suppose $\tilde{f} \in \FF(t^p,x_1,x_2)$. Then by the chain rule
$$
0 = \frac{\partial \tilde{f}}{\partial t}(t,x_1,x_2)=\frac{\partial f}{\partial x_1}(t,x_1+\bv(t)x_2)\bv'(t)x_2+\frac{\partial f}{\partial t}(t,x_1+\bv(t)x_2).
$$

By a change of variables $x_1=x_1-\bv(t)x_2$ we get:

$$
0 = \frac{\partial f}{\partial x_1}(t,x_1)\bv'(t)x_2+\frac{\partial f}{\partial t}(t,x_1).
$$
We view the above as a polynomial in variable $x_2$ over $\FF(t,x_1)$. By equating the coefficients of this polynomial to $0$ we get the following two equations
\begin{equation}\label{equation single condition by equating free coeficient}
\frac{\partial f}{\partial t}(t,x_1)=0
\end{equation}
\begin{equation}\label{equation single condition by equating linear coeficient}
\frac{\partial f}{\partial x_1}(t,x_1)\bv'(t)=0
\end{equation}

Since $\bv'(t) \not = 0$ by our assumption that $\frac{\bv}{\av}=\bv \not \in \FF(t^p)$, Eq.~\ref{equation single condition by equating linear coeficient} holds if and only if $\frac{\partial f}{\partial x_1} = 0$. Hence the two equations, Eq.~\ref{equation single condition by equating free coeficient} and Eq.~\ref{equation single condition by equating linear coeficient}, hold if and only if both derivatives by $x_1$ and by $t$ vanish. Equivalently, $f \in \FF(t^p,x_1^p)$.
\end{proof}

\begin{lem} \label{lem single variable low char - image of f is separable in t}
Let $f \in \FF[t][x_1]$ be \sqf{} in $\overline{\FF}[t][x_1]$. Let $\av,\bv,\cv \in \FF[t]$ such that $\gcd(\av,\bv)=1$. Let $l \in \FF[t][x_1,x_2]$ defined by $l(t,x_1,x_2):=\av(t)x_1+\bv(t)x_2+\cv(t)$. Let $\Psi_l : \FF[t][x_1] \rightarrow \FF[t][x_1,x_2]$ be the $\FF[t]$-homomorphism defined by $x_1 \mapsto l$. If at least one of the following holds then $\Psi_l(f)$ is separable in $t$.
\begin{enumerate}
\item \label{item single variable theorem - large char}$p=0$ or $p > \|(\av,\bv,\cv)\|\deg_{x_1} f+\deg_t f$.
\item \label{item single variable theorem - arbitrary char}$\frac{\bv}{\av} \not \in \FF(t^p)$ where $\av \not = 0$.
\end{enumerate}
\end{lem}

\begin{proof}
Let $f=\prod_{i=1}^k f_i$ be a factorization of $f$ into irreducible factors. Then by homomorphism properties
\begin{equation}\label{equation single varialble low cahr - factorizatoin of rho(f)}
\Psi_l(f) = \prod_{i=1}^k \Psi_{l}(f_i).
\end{equation}

By Lemma~\ref{lem single variable large char - f is irr then so is rho(f)} part~\ref{single variable large char - is irreducible} Eq.~\ref{equation single varialble low cahr - factorizatoin of rho(f)} provides a factorization of $\Psi_l(f)$ into irreducible factors in $\FF[t][x_1,x_2]$. We now show that $\Psi_l(f)$ is \sqf{}. Suppose $\Psi_l(f_i)=\Psi_l(f_j)\alpha$ where $i \not = j$ and $\alpha \in \FF$. Since $\Psi_l(\alpha)=\alpha$ we have $\Psi_l(f_i)=\Psi_l(f_j)\Psi_l(\alpha)=\Psi_l(f_j\alpha)$. Since $\Psi_l$ is injective by Lemma~\ref{lem single variable large char - f is irr then so is rho(f)} part~\ref{single variable large char - can be extended}, we conclude that $f_i=f_j\alpha$. But that is a contradiction to the assumption that $f$ is \sqf{}. Hence $\Psi_l(f)$ is \sqf{}. If $p=0$ then we are done, since by Corollary~\ref{cor if f is separable in t it has a factor in ftp} part~\ref{part if char=0 then f is separable} $\Psi_l(f)$ being \sqf{} implies that $\Psi_l(f)$ is separable in $t$.

Now suppose $p>0$, and suppose on the contrary that $\Psi_l(f)$ is not separable as polynomial in $t$. Then by Corollary~\ref{cor if f is separable in t it has a factor in ftp} part~\ref{part part if char>0 then f has a factor in ftp} $\Psi_l(f)$ has an irreducible factor in $\FF[t^p][x_1,x_2]$, where its degree in $t$ is not zero. Thus we can assume without loss of generality that for some $\btwn{i}{1}{k}$

\begin{equation} \label{equation single variable - f_i in ftp}
\Psi_l(f_i) \in \FF[t^p][x_1,x_2],\mbox{ where } \deg_t \Psi_l(f_i) > 0.
\end{equation}

We will now show that Eq.~\ref{equation single variable - f_i in ftp} cannot hold. The proof splits here, depending on which of (\ref{item single variable theorem - large char}) or (\ref{item single variable theorem - arbitrary char}) of the lemma holds.

Suppose (\ref{item single variable theorem - large char}) holds. Then Eq.~\ref{equation single variable - f_i in ftp} cannot hold since for any $\btwn{i}{1}{k}$
$$
\deg_t (\Psi_l(f_i)) \leq \deg_t (\Psi_l(f)) \leq \|(\av,\bv,\cv)\| \deg_{x_1} f+\deg_t f < p
$$
where the last inequality is by our assumption on $p$.

Now suppose (\ref{item single variable theorem - arbitrary char}) holds. Since $\frac{\bv}{\av} \not \in \FF(t^p)$, by Lemma~\ref{lem single variable low char - separable displacments} $f_i \in \FF[t^p,x_1^p]$. Now by the equivalent conditions in Theorem~\ref{thm absolutely square-free equivalent conditions} $f$ cannot be \sqf{} in $\overline{\FF}[t,x_1]$. This is a contradiction to our assumptions which shows that Eq.~\ref{equation single variable - f_i in ftp} does not hold. Hence $\Psi_l(f)$ is separable in $t$ as was to show.
\end{proof}

We now prove Theorem~\ref{thm single variable - algebraic}.

\begin{proof}
Let $\tilde{f}:=\Psi_l(f)$. We prove this by applying Lemma~\ref{lem main - squarefree values of polynomial separable in t} to $\tilde{f}$. By Lemma~\ref{lem main - squarefree values of polynomial separable in t} there exists a polynomial $P \in \FF[x_1,x_2]$ such that
\begin{multline}\label{equation proof polynomial property 1}
\{(\beta_1,\beta_2) \in \FF^2: \tilde{f}(t,\beta_1,\beta_2) \mbox{ is not separable}\} \subseteq \\\{(\beta_1,\beta_2) \in \FF^2: P(\beta_1,\beta_2)=0\}.
\end{multline}
But
\begin{equation}\label{equation proof polynomial property 2}
\tilde{f}(t,\beta_1,\beta_2) = f(t,\av(t)\beta_1+\bv(t)\beta_2+\cv(t)).
\end{equation}
From Eq.~\ref{equation proof polynomial property 1} and Eq.~\ref{equation proof polynomial property 2} it follows that Eq.~\ref{equation polynomial P property} holds as required.

Let $C(N)$ be sufficiently large such that $C(N) \geq \|(\av,\bv,\cv)\|\deg_{x} f+\deg_t f$ holds. Assume at least one of (\ref{part thm sigle variable - albgebraic - single variable perturbation}) and (\ref{part thm sigle variable - albgebraic - two variables perturbation}) in Theorem~\ref{thm single variable - algebraic} holds. It follows that at least one of (\ref{item single variable theorem - large char}) and (\ref{item single variable theorem - arbitrary char}) in Lemma~\ref{lem single variable low char - image of f is separable in t} holds. Thus by Lemma~\ref{lem single variable low char - image of f is separable in t} $\tilde{f}$ is separable in $t$. Hence by Lemma~\ref{lem main - squarefree values of polynomial separable in t} $P$ is non-zero.

Now,
$$\deg_t \tilde{f} \leq \|(\av,\bv,\cv)\| \deg_x f + \deg_t f$$
and
$$\deg_{\vec{x}} \tilde{f} \leq \deg_x f,$$
where $\deg_{\vec{x}}$ denotes the total degree of $\tilde{f}$ as polynomial in variables $x_1,x_2$. Assigning this into the bound which is given by Eq.~\ref{equation single variable main lemma pol bound} in Lemma~\ref{lem main - squarefree values of polynomial separable in t} we get
$$
\deg P \leq 2(\deg_t \tilde{f}-1)\deg_{\vec{x}} \tilde{f} \leq
$$
$$
(2(\|(\av,\bv,\cv)\|\deg_x f + \deg_t f)-1)\deg_x f.
$$
Let $\tilde{C}(N)$ be such that $\tilde{C}(N) \geq (2(\|(\av,\bv,\cv)\|\deg_x f + \deg_t f)-1)\deg_x f$, then $\deg P \leq \tilde{C}(N)$ holds. Since $P$ depends on $\tilde{f}$, and hence on $\av,\bv,\cv$ and $f$, we may denote it by $P_{f,\av,\bv,\cv}$.
\end{proof}

\section{Square free values of multivariate polynomials}\label{section simple gen}
The main result of this section is a generalization of Theorem~\ref{thm perturbations of a single variable polynomial finit field} to multivariate polynomials. The proof of the main result is a direct generalization of the proof of Theorem~\ref{thm perturbations of a single variable polynomial finit field}. We use this generalization to estimate the number of \sqf{} values of a multivariate polynomial $f$ at the set $\monic_{m_1}\times \dots \times \monic_{m_d}$ where $m_1,\dots,m_d \in \N$, $\deg_t f,\deg_{\vec{x}} f$ are fixed and $q \to \infty$. This result is stated in Corollary~\ref{cor sqf values in of monic polyomials} in Section~\ref{subsection sqf values of a multivariate polynomial}.

We first fix more notation for this section. Let $D$ be a unique factorization domain. For each $\btwn{i}{1}{d}$, let $\av_i,\bv_i,\cv_i \in D$. Let $l_i \in D[x_i,x_{d+i}]$ defined by $l_i = \av_ix_i+\bv_ix_{d+i}+\cv_i$. We define a $D$- homomorphism $\Psi_{l} : \polring{D}{d}{x} \rightarrow \polring{D}{2d}{x}$ by $x_i \mapsto l_i$ for each $\btwn{i}{1}{d}$. Equivalently, $\Psi_l$ is defined by $f(x_1,\dots,x_d) \mapsto f(l_1(x_1,x_{d+1}),\dots,l_d(x_d,x_{2d}))$. Throughout this section $\av_1,\bv_1,\cv_1,\dots,\av_d,\bv_d,\cv_d$ denote elements in $D$ and $\Psi_{l}$ is the homomorphism as defined above. In the first lemma $D$ is assumed to be any unique factorization domain, while in the rest of this section $D$ is assumed to be $\FF[t]$.

To short the notation, we use vector notation. Hence $\vec{\av} = (\av_1,\dots,\av_d),\vec{\bv} = (\bv_1,\dots,\bv_d),\vec{\cv} = (\cv_1,\dots,\cv_d)$. Also $\vec{\beta}=(\beta_1,\dots,\beta_{2d})$ where $\beta_1,\dots,\beta_{2d} \in \FF$. For $D=\FF[t]$ we define the following notation which generalizes the corresponding notation for a univariate polynomial.
\begin{multline}
\persimple{\FF}{d}{\av}{\bv}{\cv} := \persimplelong{\FF}{d}{\av}{\bv}{\cv}=\\ \{(\av_1(t)\beta_{1}+\bv_1(t)\beta_{d+1}+\cv_1(t),\dots\\, \av_d(t)\beta_d+\bv_d(t)\beta_{2d}+\cv_{d}(t)) \in \FF[t]^d: \beta_1,\dots,\beta_{2d} \in \FF\}.
\end{multline}

\subsection{Lemmas needed for the proof of the main result}
We start by generalizing the lemmas of Section~\ref{subsection Proof of Theorem thm single variable - algebraic}.

\begin{lem}[Generalization of Lemma~\ref{lem single variable large char - f is irr then so is rho(f)} part~\ref{single variable large char - is irreducible}]\label{lem simple gen lem single variable large char - f is irr then so is rho(f)}
Let $D$ be a unique factorization domain. Let $\bv_1,\av_1,\cv_1,\dots, \bv_d,\av_d,\cv_d \in D$, such that for each $\btwn{i}{1}{d}$, $\gcd(\av_i,\bv_i)=1$, and let $\Psi_l: D[x_1,\dots,x_d] \rightarrow D[x_1,\dots, x_{2d}]$ be a $D$-homomorphism defined as in the beginning of Section~\ref{section simple gen}. If $f$ is irreducible in $D[x_1,\dots,x_d]$ then $\Psi_l(f)$ is irreducible in $D[x_1,\dots,x_{2d}]$.
\end{lem}

\begin{proof}
The proof follows by applying Lemma~\ref{lem single variable large char - f is irr then so is rho(f)} part~\ref{single variable large char - is irreducible} inductively. The case where $d=1$ is proved by Lemma~\ref{lem single variable large char - f is irr then so is rho(f)} part~\ref{single variable large char - is irreducible}. Suppose the lemma holds for $d-1$. Let $\tilde{f} \in D[x_1,\dots,x_{2d-1}]$ defined by \begin{multline}\tilde{f}=f(\av_1x_1+\bv_1x_{d+1}+\cv_1,\av_2x_2+\bv_2x_{d+2}+\cv_2,\dots\\,\av_{d-1}x_{d-1}+\bv_{d-1}x_{2d-1}+\cv_{d-1},x_d).
\end{multline}
We first view $f$ as a polynomial in the variables $x_1,\dots,x_{d-1}$ over the domain $D[x_d]$. Then by the induction assumption $\tilde{f}$ is irreducible. Now view $\tilde{f}$ as a univariate polynomial in the variable $x_d$. Then by Lemma~\ref{lem single variable large char - f is irr then so is rho(f)} part~\ref{single variable large char - is irreducible} it follows that $\tilde{f}(x_1,\dots,x_{d-1},\av_dx_d+\bv_{d}x_{2d}+\cv_d)$ is irreducible. But $\Psi_l(f)=\tilde{f}(x_1,\dots,x_{d-1},\av_dx_d+\bv_{d}x_{2d}+\cv_d)$. Hence the lemma follows.
\end{proof}

\begin{lem}[Generalization of Lemma~\ref{lem single variable low char - separable displacments}]\label{lem simple gen low char - separable displacments}
Let $\FF$ be a field of positive characteristic $p$. Let $f \in \polfield{\FF(t)}{d}{x}$. Let $\av_1,\bv_1,\cv_1\dots,\av_d,\bv_d,\cv_d \in \FF(t)$ such that for each $\btwn{i}{1}{d}$ $\av_i \not = 0$ and $\frac{\bv_i}{\av_i} \not \in \FF(t^p)$. Suppose \begin{multline}
f(t,\av_1(t)x_1+\bv_1(t)x_{d+1}+\cv_{1}(t),\dots\\,\av_d(t)x_d+\bv_d(t)x_{2d}+\cv_{d}(t)) \in \FF(t^p,x_1,\dots,x_{2d}),
\end{multline}
then $f \in \FF(t^p,x_1^p,\dots,x_d^p)$.
\end{lem}

\begin{proof}
As in the proof of Lemma~\ref{lem single variable low char - separable displacments}, we first note that it is sufficient to prove the lemma in the case where for each $\btwn{i}{1}{d}$ $\av_i=1$ and $\cv_i=0$. This is because if we define $\overline{f}(t,x_1,\dots,x_d):=f(t,\av_1(t)x_1+\cv_1(t),\dots,\av_d(t)x_d+\cv_d(t))$, and define $\overline{\av_i}:=1,\overline{\bv_i}:=\frac{\bv_i}{\av_i},\overline{\cv_i}:=0$, then $f(t,\av_1(t)x_1+\bv_1(t)x_{d+1}+\cv_1(t),\dots,\av_d(t)x_d+\bv_d(t)x_{2d}+\cv_d(t))=\overline{f}(t,x_1+\frac{\bv_1}{\av_1}(t)x_{d+1},\dots,x_d+\frac{\bv_d}{\av_d}(t)x_{2d})
=\overline{f}(t,\overline{\av_1}(t)x_1+\overline{\bv_1}(t)x_{d+1}+\overline{\cv_1}(t),\dots,\overline{\av_d}(t)x_d+\overline{\bv_d}(t)x_{2d}+\overline{\cv_d}(t))$. Proving the lemma for the special case where $f=\overline{f}$ and for each $\btwn{i}{1}{d}$ $\av_i=\overline{\av_i},\bv_i=\overline{\bv_i},\cv_i=\overline{\cv_i}$ will show that if

$$\overline{f}(t,x_1+\frac{\bv_1}{\av_1}(t)x_{d+1},\dots,x_d+\frac{\bv_d}{\av_d}(t)x_{2d}) \in
\polfield{\FF(t^p)}{d}{x^p}$$ then$$\overline{f} \in \polfield{\FF(t^p)}{d}{x^p},$$i.e. $\overline{f}=g(t^p,x_1^p,\dots,x_d^p)$ for some $g \in \polfield{\FF(t)}{d}{x}$. But then \begin{multline}f\left (t,x_1,\dots,x_d\right )=\overline{f}\left (t,\frac{x_1-\cv_1(t)}{\av_1(t)},\dots,\frac{x_d-\cv_d(t)}{\av_d(t)}\right)=\\g\left (t^p, \left (\frac{x_1-\cv_1(t)}{\av_1(t)}\right )^p,\dots,\left (\frac{x_d-\cv_d(t)}{\av_d(t)}\right )^p\right ) \in \polfield{\FF(t^p)}{d}{x^p}.
\end{multline}

Hence from now on we can assume that for each $\btwn{i}{1}{d}$ $\av_i=1$ and $\cv_i=0$. Now, let$$\tilde{f}:=f(t,x_1+\bv_1(t)x_{d+1},\dots,x_d+\bv_d(t)x_{2d}).$$Suppose $\tilde{f} \in \FF(t^p)(x_1,\dots,x_{2d})$. Then by the chain rule

\begin{multline}
0 = \frac{\partial \tilde{f}}{\partial t}(t,x_1,\dots,x_d)=\\
\sum_{i=1}^d \partiald{f}{x_i}(t,x_1+\bv_1(t)x_{d+1},\dots,x_d+\bv_{d}(t)x_{2d})\bv_i'(t)x_{d+i}+\\\partiald{f}{t}(t,x_1+\bv_1(t)x_{d+1},\dots,x_d+\bv_{d}(t)x_{2d}).
\end{multline}
By a change of variables $x_i=x_i-\bv_i(t)x_{d+i},\forall \btwn{i}{1}{d}$ we get:
$$
0 = \sum_{i=1}^d \frac{\partial f}{\partial x_i}(t,x_1,\dots,x_d)\bv_i'(t)x_{d+i}+\frac{\partial f}{\partial t}(t,x_1,\dots,x_d).
$$
We view the above as a polynomial in variables$$x_{d+1},\dots,x_{2d}$$over $\polfield{\FF(t)}{d}{x}$. By equating the coefficients of this polynomial to $0$ we get the following equations
\begin{equation}\label{equation simple gen single condition by equating free coeficient}
\frac{\partial f}{\partial t}(t,x_1,\dots,x_d)=0
\end{equation}
\begin{equation}\label{equation simple gen single condition by equating linear coeficient}
\partiald{f}{x_i}(t,x_1,\dots,x_d)\bv_i'(t)=0,\;\forall\btwn{i}{1}{d}
\end{equation}

The lemma follows from Eq.~\ref{equation simple gen single condition by equating free coeficient} and Eq.~\ref{equation simple gen single condition by equating linear coeficient} above and by the assumption that $\frac{\bv_i}{\av_i}=\bv_i \not \in \FF(t^p)$, hence $\bv_i' \not = 0$, for each $\btwn{i}{1}{d}$.


\end{proof}


\begin{lem}[Generalization of Lemma~\ref{lem single variable low char - image of f is separable in t}]\label{lem simple gen low char - image of f is separable in t}
Let $f \in \polring{\FF[t]}{d}{x}$ be \sqf{} in $\polring{\overline{\FF}[t]}{d}{x}$. Let $\av_1,\bv_1,\cv_1,\dots,\av_d,\bv_d,\cv_d \in \FF[t]$ such that for each $\btwn{i}{1}{d}$, $\gcd(\av_i,\bv_i)=1$. Let $\Psi_l : \polring{\FF[t]}{d}{x} \rightarrow \polring{\FF[t]}{2d}{x}$ be the $\FF[t]$-homomorphism as defined in the beginning of Section~\ref{section simple gen}. If at least one of the following holds then $\Psi_l(f)$ is separable in $t$.
\begin{enumerate}
\item \label{item simple gen theorem - large char}$p=0$ or 
    $p > \|(\bv_1,\av_1,\cv_1,\dots,\av_d,\bv_d,\cv_d)\|\deg_{\vec{x}}f+\deg_t f$.
\item \label{item simple gen theorem - arbitrary char}For each $\btwn{i}{1}{d}$ $\frac{\bv_i}{\av_i} \not \in \FF(t^p)$ where $\av_i \not = 0$.
\end{enumerate}
\end{lem}
\begin{proof}
Let $f=\prod_{i=1}^k f_i$ be a factorization of $f$ into irreducible factors. Then by homomorphism properties
\begin{equation}\label{equation simple gen low cahr - factorizatoin of rho(f)}
\Psi_l(f) = \prod_{i=1}^k \Psi_{l}(f_i).
\end{equation}

By Lemma~\ref{lem simple gen lem single variable large char - f is irr then so is rho(f)} Eq.~\ref{equation simple gen low cahr - factorizatoin of rho(f)} provides a factorization of $\Psi_l(f)$ into irreducible factors in $\polring{\FF[t]}{2d}{x}$. We now show that $\Psi_l(f)$ is \sqf{}. Suppose $\Psi_l(f_i)=\Psi_l(f_j)\alpha$ where $i \not = j$ and $\alpha \in \FF$. Since $\Psi_l(\alpha)=\alpha$ we have $\Psi_l(f_i)=\Psi_l(f_j)\Psi_l(\alpha)=\Psi_l(f_j\alpha)$. But $\Psi_l$ is injective, since Lemma~\ref{lem simple gen lem single variable large char - f is irr then so is rho(f)} implies in particular that the kernel of $\Psi_l$ is trivial. Thus we conclude that $f_i=f_j\alpha$. But that is a contradiction to the assumption that $f$ is \sqf{}. Hence $\Psi_l(f)$ is \sqf{}. If $p=0$ then we are done, since by Corollary~\ref{cor if f is separable in t it has a factor in ftp} part~\ref{part if char=0 then f is separable} $\Psi_l(f)$ being \sqf{} implies that $\Psi_l(f)$ is separable in $t$.

Now suppose $p>0$, and suppose on the contrary that $\Psi_l(f)$ is not separable as polynomial in $t$. Then by Corollary~\ref{cor if f is separable in t it has a factor in ftp} part~\ref{part part if char>0 then f has a factor in ftp} $\Psi_l(f)$ has an irreducible factor in $\polring{\FF[t^p]}{2d}{x}$ where its degree in $t$ is not zero. Hence we can assume without loss of generality that for some $\btwn{j}{1}{k}$

\begin{equation} \label{equation simple gen - f_i in ftp}
\Psi_l(f_j) \in \polring{\FF[t^p]}{2d}{x}\mbox{, where }\deg_t \Psi_l(f_j) > 0.
\end{equation}

We will now show that Eq.~\ref{equation simple gen - f_i in ftp} cannot hold. We split the proof, depending on which of (\ref{item simple gen theorem - large char}) or (\ref{item simple gen theorem - arbitrary char}) holds.

Suppose (\ref{item simple gen theorem - large char}) holds. Then Eq.~\ref{equation simple gen - f_i in ftp} cannot hold since for any $\btwn{j}{1}{k}$
$$
\deg_t (\Psi_l(f_j)) \leq \deg_t (\Psi_l(f)) \leq \|(\bv_1,\av_1,\cv_1,\dots,\av_d,\bv_d,\cv_d)\|\deg_{\vec{x}}f+\deg_t f < p
$$
where the last inequality is by our assumption on $p$.

Suppose (\ref{item simple gen theorem - arbitrary char}) holds. Since $\frac{\bv_i}{\av_i} \not \in \FF(t^p)\;\forall \btwn{i}{1}{d}$, by Lemma~\ref{lem simple gen low char - separable displacments} $f_j \in \FF[t^p,x_1^p,\dots,x_d^p]$. Now by the equivalent conditions in Theorem~\ref{thm absolutely square-free equivalent conditions} $f$ cannot be \sqf{} in $\overline{\FF}[t,x_1,\dots,x_d]$. This is a contradiction to our assumptions which shows that Eq.~\ref{equation simple gen - f_i in ftp} does not hold. Hence $\Psi_l(f)$ is separable in $t$.
\end{proof}

\subsection{The main theorem for a multivariate polynomial}
We now state the main theorem for a multivariate polynomial over a finite field which generalizes Theorem~\ref{thm perturbations of a single variable polynomial finit field}.

\begin{thm}[Generalization of Theorem~\ref{thm perturbations of a single variable polynomial finit field}]\label{thm simple gen perturbations of a polynomial finit field}
Let $f \in \polring{\fq[t]}{d}{x}$ be a \sqf{} polynomial. Let $\av_1,\bv_1,\cv_1,\dots,\av_d,\bv_d,\cv_d \in \fq[t]$ such that for each $\btwn{i}{1}{d}$ $\gcd(\av_i,\bv_i)=1$. Let $N \in \N$. Assume $$\deg_{\vec{x}} f,\deg_t f, \|(\av_1,\bv_1,\cv_1,\dots,\av_d,\bv_d,\cv_d)\| \leq N.$$ Assume that at least one of the following holds

\begin{enumerate}
\item \label{part simple gen theorem large char}$p > C(N)$ where $C(N) \in \N$ is a constant which depends only on $N$.
\item\label{part simple gen theorem arbitrary char}For each $\btwn{i}{1}{d}$ $\frac{\bv_i}{\av_i} \not \in \FF(t^p)$ where $\av_i \not = 0$.

Then while $N$ remains fixed, we have:
\begin{equation}\label{equation simiple gen main theorem estimate}
   \frac{\#(\squaref{\fq}{d}{f} \bigcap \persimple{\fq}{d}{\av}{\bv}{\cv})}{\#\persimple{\fq}{d}{\av}{\bv}{\cv})} = 1 + O \left (\frac{1}{q} \right )\;,\quad \mbox{as } q\to \infty.
\end{equation}
In particular, if $q$ is sufficiently large with respect to $N$ then there exist $\vec{\beta} \in \fq^{2d}$ such that $f(t,\cv_1(t)+\av_1(t)\beta_1+\bv_1(t)\beta_{d+1},\dots,\cv_d(t)+\av_d(t)\beta_d+\bv_d(t)\beta_{2d})$ is \sqf{}.
\end{enumerate}
\end{thm}

Let $\kappa_1,\dots,\kappa_d \in \N$ where for each $\btwn{i}{1}{d}$ $\kappa_i \not = 0 \mod p$. Example~\ref{example simple perturbation generalization} in Section~\ref{subsection sqf values of a multivariate polynomial} is the specific case of Theorem~\ref{thm simple gen perturbations of a polynomial finit field} where $\av_i=1$, and $\bv_i=t^{\kappa_i}$, $\forall \btwn{i}{1}{d}$.

As we did in the case of a univariate polynomial, we prove Theorem~\ref{thm simple gen perturbations of a polynomial finit field} by stating and proving an algebraic theorem which holds for a general field $\FF$, and showing that Theorem~\ref{thm simple gen perturbations of a polynomial finit field} follows from the algebraic theorem in the case where $\FF$ is a finite field. The following theorem holds for any field $\FF$.

\begin{thm}[Generalization of Theorem~\ref{thm single variable - algebraic}]\label{thm simple gen - algebraic}
Let $\FF$ be a field, and let $\overline{\FF}$ be an algebraic closure of $\FF$. Let $f \in \polring{\FF[t]}{d}{x}$ be a polynomial which is \sqf{} in $\polring{\overline{\FF}[t]}{d}{x}$. Let $\av_1,\bv_1,\cv_1,\dots,\av_d,\bv_d,\cv_d \in \FF[t]$ such that for each $\btwn{i}{1}{d}$ $\gcd(\av_i,\bv_i)=1$. Assume$$\deg_{\vec{x}} f,\deg_t f, \|(\av_1,\bv_1,\cv_1,\dots,\av_d,\bv_d,\cv_d)\| \leq N.$$Then there exists a polynomial $\multP{\av}{\bv}{\cv} \in \polring{\FF}{2d}{x}$ which depends on $\bv_1,\av_1,\cv_1,\dots,\bv_d,\av_d,\cv_d$ and $f$ such that
\begin{multline}\label{equation simple gen polynomial P property}
\{\vec{\beta} \in \FF^{2d} :\\ f(t,\av_1(t)\beta_1+\bv_1(t)\beta_{d+1}+\cv_1(t),\dots,\av_d(t)\beta_d+\bv_d(t)\beta_{2d}+\cv_d(t)) \mbox{ is not separable}\}\\ \subseteq \{\vec{\beta} \in \FF^{2d} : \multP{\av}{\bv}{\cv}(\beta_1,\dots,\beta_{2d})=0\}.
\end{multline}
Moreover, there exists a constant $\tilde{C}(N)$ which depends only on $N$ such that$$\deg \multP{\av}{\bv}{\cv} \leq \tilde{C}(N).$$
$\multP{\av}{\bv}{\cv}$ is non-zero if at least one of the following holds:
\begin{enumerate}
\item \label{part thm simple gen - albgebraic - single variable perturbation}
$p=0$ or $p>C(N)$ where $C(N)$ is a constant which depends only on $N$.
\item\label{part thm simple gen - albgebraic - two variables perturbation}
For each $\btwn{i}{1}{d}$ $\frac{\bv_i}{\av_i} \not \in \FF(t^p)$ where $\av_i \not = 0$.
\end{enumerate}
\end{thm}

\begin{proof}
Let $\tilde{f} \in \FF[x_1,\dots,x_{2d}], \tilde{f}:=\Psi_l(f)$. We prove this by applying Lemma~\ref{lem main - squarefree values of polynomial separable in t} to $\tilde{f}$. By Lemma~\ref{lem main - squarefree values of polynomial separable in t} there exists a polynomial $P \in \FF[x_1,\dots,x_{2d}]$ such that
\begin{multline}\label{equation simple gen polynomial P property proof 1}
\{\vec{\beta} \in \FF^{2d}: \tilde{f}(t,\beta_1,\dots,\beta_{2d}) \mbox{ is not separable}\} \\ \subseteq \{\vec{\beta} \in \FF^{2d}: P(\beta_1,\dots,\beta_{2d})=0\}.
\end{multline}
But
\begin{multline}\label{equation simple gen polynomial P property proof 2}
\tilde{f}(t,\beta_1,\dots,\beta_{2d})=f(t,\av_1(t)\beta_1+\bv_1(t)\beta_{d+1}+\cv_1(t),\dots\\,\av_d(t)\beta_d+\bv_d(t)\beta_{2d}+\cv_d(t)).
\end{multline}
From Eq.~\ref{equation simple gen polynomial P property proof 1} and Eq.~\ref{equation simple gen polynomial P property proof 2} it follows that Eq.~\ref{equation simple gen polynomial P property} holds as required.

Let $C(N)$ be such that $C(N) \geq \|(\bv_1,\av_1,\cv_1,\dots,\av_d,\bv_d,\cv_d)\|\deg_{\vec{x}}f+\deg_t f$. Suppose at least one of (\ref{part thm simple gen - albgebraic - single variable perturbation}) or (\ref{part thm simple gen - albgebraic - two variables perturbation}) in Theorem~\ref{thm simple gen - algebraic} holds. It follows that at least one of (\ref{item simple gen theorem - large char}) or (\ref{item simple gen theorem - arbitrary char}) in Lemma~\ref{lem simple gen low char - image of f is separable in t} holds. Hence by Lemma~\ref{lem simple gen low char - image of f is separable in t} $\tilde{f}$ is separable in $t$. Hence by Lemma~\ref{lem main - squarefree values of polynomial separable in t} $P$ is non-zero.

Now,
$$
\deg_t \tilde{f} \leq \|(\av_1,\bv_1,\cv_1,\dots,\av_d,\bv_d,\cv_d)\| \deg_{\vec{x}} f + \deg_t f
$$
and
$$
\deg_{\vec{x}} \tilde{f} \leq \deg_{\vec{x}} f.
$$

Assigning this into the bound which is given by Eq.~\ref{equation single variable main lemma pol bound} in Lemma~\ref{lem main - squarefree values of polynomial separable in t} we get
\begin{multline}
\deg P \leq 2(\deg_t \tilde{f}-1)\deg_{\vec{x}} \tilde{f} \leq
\\
(2(\|(\av_1,\bv_1,\cv_1,\dots,\av_d,\bv_d,\cv_d)\|\deg_{\vec{x}} f + \deg_t f)-1)\deg_{\vec{x}} f.
\end{multline}

Let $\tilde{C}(N)$ be such that $\tilde{C}(N) \geq (2(\|(\av_1,\bv_1,\cv_1,\dots,\av_d,\bv_d,\cv_d)\|\deg_{\vec{x}} f + \deg_t f)-1)\deg_{\vec{x}} f$, then $\deg P \leq \tilde{C}(N)$. Since $P$ depends on $\tilde{f}$, and hence on $\av_1,\bv_1,\cv_1,\dots,\av_d,\bv_d,\cv_d$ and $f$, we may denote it by $\multP{\av}{\bv}{\cv}$.
\end{proof}

We now show that Theorem~\ref{thm simple gen perturbations of a polynomial finit field} follows from Theorem~\ref{thm simple gen - algebraic} in the case where $\FF$ is a finite field.

\begin{proof}
First, over a finite field the requirement in Theorem~\ref{thm simple gen - algebraic} that $f$ is \sqf{} in $\polring{\overline{\FF}[t]}{d}{x}$ can be replaced by the requirement that $f$ is \sqf{}. Over a finite field the two requirements are equivalent by Corollary~\ref{cor square-free equals absolutely square-free}, since $\fq$ is a perfect field. Likewise, by the same corollary over a finite field $f(t,\av_1(t)\beta_1+\bv_1(t)\beta_{d+1}+\cv_1(t),\dots,\av_d(t)\beta_d+\bv_d(t)\beta_{2d}+\cv_d(t))$ being separable and $f(t,\av_1(t)\beta_1+\bv_1(t)\beta_{d+1}+\cv_1(t),\dots,\av_d(t)\beta_d+\bv_d(t)\beta_{2d}+\cv_d(t))$ being \sqf{} can be used interchangeably, since the two are equivalent over a perfect field.

With the assumptions and definitions as in Theorem~\ref{thm simple gen perturbations of a polynomial finit field}, we now show why the estimate in Eq.~\ref{equation simiple gen main theorem estimate} of Theorem~\ref{thm simple gen perturbations of a polynomial finit field} follows. The error $E(q)$ of estimating
$$\frac{\#(\squaref{\fq}{d}{f} \bigcap \persimple{\fq}{d}{\av}{\bv}{\cv})}{\#\persimple{\fq}{d}{\av}{\bv}{\cv}}$$
by $1$ is given by
\begin{equation}
\begin{split}
E(q) =\frac{\#(\squaref{\fq}{d}{f}^c \bigcap \persimple{\fq}{d}{\av}{\bv}{\cv})}{\#\persimple{\fq}{d}{\av}{\bv}{\cv}}= \\\frac{1}{q^{2d}}\#\{\vec{\beta} \in \fq^{2d} : f(t,\av_1(t)\beta_1+\bv_1(t)\beta_{d+1}+\cv_1(t),\dots\\,\av_d(t)\beta_d+\bv_d(t)\beta_{2d}+\cv_d(t)) \mbox{ is not \sqf{}}\} \leq\\
\frac{\#\{\vec{\beta} \in \fq^{2d} : \multP{\av}{\bv}{\cv}(\beta_1,\dots,\beta_{2d})=0\}}{q^{2d}}.
\end{split}
\end{equation}

Assume first that $\av_1,\bv_1,\cv_1,\dots,\av_d,\bv_d,\cv_d$ are such that $\multP{\av}{\bv}{\cv}$ is nonzero. Then applying Lemma~\ref{lem elementary bound on number of zeros}
$$
E(q) \leq \frac{\tilde{C}(N)q^{2d-1}}{q^{2d}}.
$$
Hence keeping $N$ fixed while $q \to \infty$, $E(q)=O \left (\frac{1}{q} \right )$.

It remains to show why it follows from the assumptions of Theorem~\ref{thm simple gen perturbations of a polynomial finit field} that $\multP{\av}{\bv}{\cv}$ is non-zero. The latter is true since we assume at least one of (\ref{part simple gen theorem large char}) and (\ref{part simple gen theorem arbitrary char}) in Theorem~\ref{thm simple gen perturbations of a polynomial finit field} holds, but by letting $C(N)$ in Theorem~\ref{thm simple gen perturbations of a polynomial finit field} be the same as $C(N)$ in Theorem~\ref{thm simple gen - algebraic}, (\ref{part simple gen theorem large char}) and (\ref{part simple gen theorem arbitrary char}) in Theorem~\ref{thm simple gen perturbations of a polynomial finit field} are the same as (\ref{part thm simple gen - albgebraic - single variable perturbation}) and (\ref{part thm simple gen - albgebraic - two variables perturbation}) in Theorem~\ref{thm simple gen - algebraic}.
\end{proof}



\subsection{Proof of Corollary~\ref{cor sqf values in of monic polyomials}}

We now show that Corollary~\ref{cor sqf values in of monic polyomials} which was stated in Section~\ref{subsection sqf values of a multivariate polynomial} follows from Theorem~\ref{thm simple gen perturbations of a polynomial finit field}. For each choice of $\cv_1,\dots,\cv_d$ where $\cv_1,\dots,\cv_d$ are monic with the first two coefficients zero, we use Theorem~\ref{thm simple gen perturbations of a polynomial finit field} to estimate the number of \sqf{} values of $f$ which are obtained by perturbing the first two coefficients of each of $\cv_1,\dots,\cv_d$. By summing over all possible choices for $\cv_1,\dots,\cv_d$, we get the result which is stated in Corollary~\ref{cor sqf values in of monic polyomials}. We now show this in more details.

\begin{proof}
For each $\btwn{i}{1}{d}$ let $C_i$ be the set of monic polynomials of degree $m_i$ with the first two coefficients zero. Explicitly,
$$C_i = \left \{\cv \in \fq[t]: \cv(t) = t^{m_i}+\sum_{j=2}^{m_i-1}a_i t^j, a_2\dots,a_{m_i-1} \in \fq\right \}.
$$


Let $\cv_1 \in C_1,\dots,\cv_d \in C_d$, and for each $\btwn{i}{1}{d}$ let $\av_i=t,\bv_i=1$. Then
\begin{multline}\persimple{\fq}{d}{\av}{\bv}{\cv}=\\
\{(t\beta_{1}+\beta_{d+1}+\cv_1(t),\dots, t\beta_d+\beta_{2d}+\cv_{d}(t)) \in \FF[t]^d : \beta_1,\dots,\beta_{2d} \in \fq\}.
\end{multline}


For each choice of $\vec{\cv} \in \FF[t]^d$ where $\cv_1 \in C_1,\dots,\cv_d \in C_d$, $\persimple{\fq}{d}{\av}{\bv}{\cv}$ is the set of polynomials obtained by perturbing the first two coefficients of $\cv_1,\dots,\cv_d$. Each element of $\monic_{m_1}\times \dots \times \monic_{m_d}$ is an element of $\persimple{\fq}{d}{\av}{\bv}{\cv}$ for exactly one choice of $\vec{\cv}$. Stating this differently
$$
\monic_{m_1}\times \dots \times \monic_{m_d} = \bigcup_{\cv_1 \in C_1,\dots,\cv_d \in C_d} \persimple{\fq}{d}{\av}{\bv}{\cv}.
$$
where the union is disjoint. Hence
$$
\#((\monic_{m_1}\times \dots \times \monic_{m_d}) \bigcap \squaref{\fq}{d}{f}) = \sum_{\cv_1 \in C_1,\dots,\cv_d \in C_d} \# (\persimple{\fq}{d}{\av}{\bv}{\cv} \bigcap \squaref{\fq}{d}{f}).
$$
The size of each $C_i$ is $\#C_i = q^{m_i-2}$. Hence the number of elements in the last sum is $q^{m-2d}$ where $m=m_1+\dots+m_d$.
Thus
\begin{multline}
\frac{\#(\squaref{\fq}{d}{f} \bigcap (\monic_{m_1} \times \dots \times \monic_{m_d}))}{\#\monic_{m_1} \times \dots \times \monic_{m_d}} =\\
 \sum_{\cv_1 \in C_1,\dots,\cv_d \in C_d} \frac{\#(\squaref{\fq}{d}{f} \bigcap \persimple{\fq}{d}{\av}{\bv}{\cv})}{\#\monic_{m_1} \times \dots \times \monic_{m_d}}=\\
\frac{1}{q^{m-2d}} \sum_{\cv_1 \in C_1,\dots,\cv_d \in C_d} \frac{\#(\squaref{\fq}{d}{f} \bigcap \persimple{\fq}{d}{\av}{\bv}{\cv})}{q^{2d}} = \frac{1}{q^{m-2d}}q^{m-2d}\left (1+O \left (\frac{1}{q} \right )\right),
\end{multline}
where the last equality follows by Theorem~\ref{thm simple gen perturbations of a polynomial finit field}
\end{proof}

\appendix

\section{Proof of Theorem~\ref{thm absolutely square-free equivalent conditions} and its corollaries}

\begin{lem}\label{lem prime ideal generator}
Let $\FF$ be a field, and $L \supseteq \FF$ an algebraic field extension of $\FF$. Let $h \in \polring{L}{d}{x}$ be nonconstant and irreducible in $\polring{L}{d}{x}$. Then there exists $h_m \in \polring{\FF}{d}{x}$ such that:
\begin{enumerate}
  \item\label{item h divides hm} $h | h_m$.
  \item\label{item hm divides any f} If $f \in \polring{\FF}{d}{x}$ and $h|f$ then $h_m$ divides $f$ in $\polring{\FF}{d}{x}$.
\end{enumerate}
\end{lem}

\begin{remark}\label{remark actual indeterminates in hm}
For any $i$, $1 \leq i \leq d$, $\deg_{x_i} h_m > 0$ implies $\deg_{x_i} h > 0$. To show this, let $h_m$ which existence is provided by Lemma~\ref{lem prime ideal generator} applied to $h$ as an element in $L[x_1,\dots,x_d]$. By applying Lemma~\ref{lem prime ideal generator} to $h$ as an element in $L[x_{i_1},\dots,x_{i_k}]$ where $x_{i_1},\dots,x_{i_k}$ are the variables that appear in $h$ we conclude that there exists $\hat{h}_m$ in $\FF[x_{i_1},\dots,x_{i_k}]$ which is divided by $h_m$. Hence the variables that appear in $h_m$ are in the set $\{x_{i_1},\dots,x_{i_k}\}$.
\end{remark}

\begin{proof}
In other words the lemma asserts that $\langle h \rangle \cap \polring{\FF}{d}{x}$ is a nonempty principal ideal in $\polring{\FF}{d}{x}$. We prove first the case where $d=1$.
Let $\alpha \in \overline{\FF}$ be a root of $h$. Since $\alpha$ is algebraic over $\FF$ there exists $g \in \FF[x_1]$ such that $g(\alpha)=0$. Since $L[x_1]$ is a principal ideal domain, there exists $r \in L[x_1]$ such that $\langle r \rangle =\langle h,g \rangle$. Also, $1 \not \in \langle r \rangle$ since $\alpha$ is a root of every polynomial in $\langle r\rangle$. Hence $r$ is not invertible. But $h$ is irreducible, hence $h|r$, and $\langle h \rangle = \langle r \rangle = \langle h,g \rangle$. Thus $g \in \langle h \rangle \cap \FF[x_1]$ and so $\langle h \rangle \cap \FF[x_1]$ is not empty.
In addition, $\langle h \rangle \cap \FF[x_1]$ is a principal ideal since $\FF[x_1]$ is a principal ideal domain. Hence the lemma follows for the case $d=1$.

Now consider the case where $d>1$. Since $h$ is nonconstant, we can assume without loss of generality that $\deg_{x_1} h > 0$. Denote
$$\tilde{\FF}=\FF(x_2,\dots,x_d),\tilde{L}=L(x_2,\dots,x_d),D=\FF[x_2,\dots,x_d],D_L:=L[x_2,\dots,x_d].$$
Then $\tilde{\FF}$ and $\tilde{L}$ are fields and $\tilde{L}$ is an algebraic extension of $\tilde{\FF}$. $h$ is irreducible also in $\tilde{L}[x_1]$ by Gauss's lemma for polynomials. It follows from the case $d=1$ that there exists $\tilde{h_m} \in \tilde{\FF}[x_1]$ such that
\begin{itemize}
\item$h|\tilde{h_m}$.
\item For any $f \in \polring{\FF}{d}{x} \subseteq \tilde{\FF}[x_1]$ such that $h|f$, $f=\tilde{h_m}\tilde{u}$ for some $\tilde{u} \in \tilde{\FF}[x_1]$.
\end{itemize}

We first show part~\ref{item hm divides any f} of the lemma. $\tilde{\FF}$ is the field of fractions of $D$. Hence we can write $\tilde{h_m}=c_{h_m} h_m$ and $\tilde{u}= c_u u$, where $c_{h_m},c_u \in \tilde{\FF}$ and $h_m, u$ are primitive polynomials in $D[x_1]$. Then
\begin{equation}\label{equation primitive polynomials multiplication}
f=\tilde{h_m}\tilde{u}=c_{h_m} h_m c_u u = c_{h_m}c_uh_mu.
\end{equation}
By Gauss's lemma for polynomials a multiplication of primitive polynomials is a primitive polynomial. Hence $h_m u$ is a primitive polynomial. Hence by Eq.~\ref{equation primitive polynomials multiplication} $c_{h_m}c_u \in D = \FF[x_2,\dots,x_d]$ or otherwise $f$ would not be a polynomial in $D[x_1]$ but a rational function. Hence by Eq.~\ref{equation primitive polynomials multiplication} $h_m$ divides $f$ in $\FF[x_1,\dots,x_d]$.

To show part \ref{item h divides hm} of the lemma, there exists $\tilde{v} \in \tilde{L}[x_1]$ such that $h\tilde{v}=\tilde{h_m}$. We can write $\tilde{v}=c_vv$ where $c_v \in \tilde{L}$ and $v \in D_L[x_1]$ is primitive. Hence $h\tilde{v}=hvc_v=h_mc_{h_m}=\tilde{h_m}$. Hence
\begin{equation}\label{equation hm existance}
hv\frac{c_v}{c_{h_m}}=h_m.
\end{equation}
But $h$ is irreducible and nonconstant in $D_{L}[x_1]$ and in particular primitive, and $v$ is primitive. Hence $hv$ is primitive by Gauss's lemma for polynomials. Hence $\frac{c_v}{c_{h_m}} \in D_L$ or otherwise by Eq.~\ref{equation hm existance} $h_m$ would not be in $D[x_1]$. Hence by Eq.~\ref{equation hm existance} $h$ divides $h_m$ in $\polring{L}{d}{x}$.

\end{proof}

\begin{lem}\label{lem square divisor in extension devides one of the facotrs}
Let $\FF$ be a field, and let $L \supseteq \FF$ be an algebraic field extension of $\FF$. Let $f \in \polring{\FF}{d}{x}$ be a polynomial which is square-free in $\polring{\FF}{d}{x}$. Let $h \in \polring{L}{d}{x}$ be an irreducible polynomial. Suppose $h^2 | f$. Let a factorization of $f$ be $f=\prod_{i=1}^{k} f_i$, where $f_1,\dots,f_k \in \polring{\FF}{d}{x}$ are irreducible as elements in $\polring{\FF}{d}{x}$. Then $h^2 | f_j$ for some $j$, $1 \leq j \leq k$. Also, any variable that appears in $f_j$ appears in $h$.
\end{lem}

\begin{proof}
$\polring{L}{d}{x}$ is a unique factorization domain. Since $h$ is a prime element in $\polring{L}{d}{x}$ and $h | f$, it follows that $h|f_j$ for some $1 \leq j \leq k$. Suppose on the contrary that $h|\frac{f}{f_j}$. By Lemma~\ref{lem prime ideal generator} there exists $h_m \in \polring{\FF}{d}{x}$ such that $h_m|f_j$ and $h_m|\frac{f}{f_j}$. But the last conclusion is a contradiction to the assumption that $f$ is \sqf{} in $\polring{\FF}{d}{x}$. Hence $h^2$ and $\frac{f}{f_j}$ are co-prime, and $h^2 | f_j$. To see that every variable that appears in $f_j$ appears in $h$, as stated in Remark~\ref{remark actual indeterminates in hm} every variable that appears in $h_m$ appears in $h$. Since $f_j$ is irreducible, hence $f_j | h_m$, every variable that appears in $f_j$ appears in $h_m$.
\end{proof}

We give the following lemma without a proof, only for reference.
\begin{lem} \label{non separable irreducible polynomial characterization}
Let $\FF$ be a field. Let $f \in \FF[x]$ be an irreducible polynomial.
\begin{enumerate}
\item \label{item separable in zero characteristic} If $\chr(\FF)=0$, then $f$ is separable.
\item \label{item separable in positive characteristic} If $\FF$ is of positive characteristic $p$, and $f$ is non-separable then $f \in \FF[x^p]$.
\end{enumerate}
\end{lem}

For a proof of part~\ref{item separable in zero characteristic} of the lemma above see Corollary 34 in chapter 13 of \cite{Dummit and Foote}. As to part~\ref{item separable in positive characteristic}, in fact, a stronger statement holds which is that there exists a unique $k \geq 0$ such that $f=f_{sep}(x^{p^k})$ where $f_{sep} \in \FF[x]$ is a separable polynomial. For a proof see proposition 38 in chapter 13 of \cite{Dummit and Foote}. For our usage the weaker statement in Lemma~\ref{non separable irreducible polynomial characterization} appears to be sufficient.

\begin{lem}\label{lem absolutely square free is squarefree in char zero}
Let $\FF$ be a field of characteristic $0$. Let $f \in \polring{\FF}{d}{x}$ be a polynomial. Then $f$ is \sqf{} in $\polring{\overline{\FF}}{d}{x}$ if and only if $f$ is \sqf{} in $\polring{\FF}{d}{x}$.
\end{lem}

\begin{proof}
On one direction, suppose $f$ has a nonconstant factor $g \in \polring{\FF}{d}{x}$ such that $g^2 | f$. Then that holds also when $g$ and $f$ are considered as elements of the larger domain $\polring{\overline{\FF}}{d}{x}$.

To see the opposite direction, suppose $f$ is \sqf{} in $\polring{\FF}{d}{x}$ and suppose on the contrary that there exists a non constant irreducible $h \in \polring{\overline{\FF}}{d}{x}$ such that $h^2$ divides $f$ in $\polring{\overline{\FF}}{d}{x}$. Let $f=\prod_{i=1}^{k} f_i$ be a factorization of $f$ where $f_i \in \polring{\FF}{d}{x}$ are irreducible polynomials in $\polring{\FF}{d}{x}$. Suppose $l \in \N$, $1 \leq l \leq d$ is such that $\deg_{x_l} h > 0$. By Lemma~\ref{lem prime ideal generator} there exists $j$, $1 \leq j \leq d$ such that $h^2$ divides $f_j$. But that means $f_j$ is not separable as polynomial in $x_l$. We now show this in more details.

Denote$$K_1=\FF(x_1,\dots,x_{l-1},x_{l+1},\dots,x_d),\;K_2=\overline{\FF}(x_1,\dots,x_{l-1},x_{l+1},\dots,x_d),$$
and denote by $\overline{K_2}$ the algebraic closure of $K_2$, which is also an algebraic closure of $K_1$. View $h$ as an element of $K_2[x_l]$ and view $f_j$ as an element of $K_1[x_l]$. Suppose $h=c\prod_{i=1}^{I}(x_l-
\alpha_i)$ where $c \in K_2$ and $\alpha_i \in \overline{K_2},\;1 \leq i \leq I$ is the factorization of $h$ into linear factors in $\overline{K_2}[x_l]$. Then in particular $(x_l-\alpha_1)^2$ divides $h^2$ and hence it divides $f_j$. Hence $f_j$ as polynomial in $K_1[x_l]$ is not separable.

But since $\chr(K_1)=0$, the last cannot hold by Lemma~\ref{non separable irreducible polynomial characterization} part~\ref{item separable in zero characteristic}. A contradiction which shows that such $h$ does not exist.
\end{proof}

We now prove Theorem~\ref{thm absolutely square-free equivalent conditions} and its two corollaries which were stated in Section~\ref{section Preliminary facts and results}.

\begin{proof}[Proof of Theorem~\ref{thm absolutely square-free equivalent conditions}]
(\ref{absolutely squarefree 1})$\Rightarrow$(\ref{absolutely squarefree 2}): Suppose $f$ is not square-free as an element in $\polring{\FF^{\frac{1}{p}}}{d}{x}$. Then there exists a non constant $h \in \polring{\FF^{\frac{1}{p}}}{d}{x}$ such that  $h^2|f$. Since $\polring{\FF^{\frac{1}{p}}}{d}{x} \subseteq \polring{\overline{\FF}}{d}{x}$, $h^2$ is a square factor of $f$ also as an element of $\polring{\overline{\FF}}{d}{x}$.
(\ref{absolutely squarefree 2})$\Rightarrow$(\ref{absolutely squarefree 3}): Assume $f$ is square-free as an element in $\polring{\FF^{\frac{1}{p}}}{d}{x}$. $\polring{\FF}{d}{x} \subseteq \polring{\FF^{\frac{1}{p}}}{d}{x}$. Hence by the same argument as in the previous part, it is immediate that $f$ is square-free as an element in $\polring{\FF}{d}{x}$. It remains to show that $f$ does not have an irreducible factor $g \in \polring{\FF}{d}{x^p}$. Suppose on the contrary there is such factor $g$. Then $g(x_1,\dots,x_d)=h(x_1^p,\dots,x_d^p)$ for some polynomial $h \in \polring{\FF}{d}{x}$

$$h(x_1,\dots,x_d)=\sum_{(e_1,\dots,e_d) \in \{0,\dots,n\}^d} c_{e_1,\dots,e_d}\prod_{j=1}^{d}x_j^{e_j}.
$$

Applying the Frobenius automorphism properties we get:
\begin{multline}\label{equation factorization by Frobenius properties}
g(x_1,\dots,x_d) =\\ \sum_{(e_1,\dots,e_d) \in \{0,\dots,d\}^n} c_{e_1,\dots,e_d}\prod_{j=1}^{d}x_{j}^{pe_j} = \left (\sum_{(e_1,\dots,e_d) \in \{0,\dots,n\}^d} c_{e_1,\dots,e_d}^{\frac{1}{p}}\prod_{j=1}^{d}x_{j}^{e_j}\right )^p
\end{multline}
Since $c_{e_1,\dots,e_d}^{\frac{1}{p}} \in \FF^{\frac{1}{p}}$ for any ${e_1,\dots,e_d}$
$$\sum_{(e_1,\dots,e_d) \in \{0,\dots,n\}^d} c_{e_1,\dots,e_d}^{\frac{1}{p}}\prod_{j=1}^{d}x_{j}^{e_j}$$
is a repeated factor of $f$ in $\polring{\FF^{\frac{1}{p}}}{d}{x}$, contradicting the assumption that $f$ is square-free as an element in $\polring{\FF^{\frac{1}{p}}}{d}{x}$.
(\ref{absolutely squarefree 3})$\Rightarrow$(\ref{absolutely squarefree 1}): Suppose $f$ is square-free as element in $\polring{\FF}{d}{x}$ but is not square-free as an element in $\polring{\overline{\FF}}{d}{x}$. Then there exists an irreducible nonconstant $h \in \polring{\overline{\FF}}{d}{x}$ such that $h^2 | f$. Let $f=\prod_{i=1}^{k} f_i$ be a factorization of $f$ where $f_i \in \polring{\FF}{d}{x}$ are irreducible polynomials in $\polring{\FF}{d}{x}$. By Lemma~\ref{lem square divisor in extension devides one of the facotrs} $h^2$ divides $f_j$ for some $j$, $1 \leq j \leq k$.

Let $l \in \N$, $1 \leq l \leq d$. Let $K_1=\FF(x_1,\dots,x_{l-1},x_{l+1},\dots,x_d)$. Suppose $\deg_{x_l} f_j > 0$. Then as stated in Lemma~\ref{lem square divisor in extension devides one of the facotrs} $\deg_{x_l} h > 0$. Since $h^2|f_j$, it follows that $f_j$ is not separable as polynomial in $K_1[x_l]$, as we showed in more details at the end of the proof of Lemma~\ref{lem absolutely square free is squarefree in char zero}. It follows by Lemma~\ref{non separable irreducible polynomial characterization} that $f_j \in K_1[x_l^p]$. If $\deg_{x_l} f_j= 0$, then $f_j \in K_1[x_l^p]$ holds as well. Hence in any case $f_j \in K_1[x_l^p]$. But that is true for any $l$, $1 \leq l \leq d$. Hence $f_j \in \FF[x_1^p,\dots,x_d^p]$.
\end{proof}

\begin{proof}[Proof of Corollary~\ref{cor square-free equals absolutely square-free}]
If $\chr(\FF)=0$, then this is stated in Lemma~\ref{lem absolutely square free is squarefree in char zero}. If $\chr(\FF)>0$, then since $\FF$ is perfect $\FF^{\frac{1}{p}}=\FF$, where $\FF^{\frac{1}{p}}$ is the field as defined in Eq.~\ref{equation ajoin roots of order p}. Hence condition~\ref{absolutely squarefree 2} of Theorem~\ref{thm absolutely square-free equivalent conditions} is equivalent to $f$ being \sqf{} in $\polring{\FF}{d}{x}$. The corollary follows by the equivalence of conditions \ref{absolutely squarefree 1} and \ref{absolutely squarefree 2} of Theorem~\ref{thm absolutely square-free equivalent conditions}.
\end{proof}

\begin{proof}[Proof of Corollary~\ref{cor if f is separable in t it has a factor in ftp}]
First, if $f$ had a square factor in $\polfield{\FF}{d}{x}[t]$ then by Gauss's lemma for polynomials it would also have a square factor in $\polring{\FF}{d}{x}[t]$. Hence we can assume $f$ is \sqf{} in $\polfield{\FF}{d}{x}[t]$.

(\ref{part if char=0 then f is separable}): View $f$ as a univariate polynomial in $t$ over $\polfield{\FF}{d}{x}$. Since $\chr(\polfield{\FF}{d}{x})=0$ in particular $\polfield{\FF}{d}{x}$ is perfect. Hence by Corollary~\ref{cor square-free equals absolutely square-free} $f$ being \sqf{} in $\polfield{\FF}{d}{x}[t]$ implies that $f$ is \sqf{} in $\overline{\polfield{\FF}{d}{x}}[t]$ where $\overline{\polfield{\FF}{d}{x}}$ denotes the algebraic closure of $\polfield{\FF}{d}{x}$. Equivalently, $f$ is separable in $t$.

(\ref{part part if char>0 then f has a factor in ftp}): View $f$ as a univariate polynomial in $t$ over $\polfield{\FF}{d}{x}$. Then $f$ is \sqf{} as an element in $\FF(x_1,\dots,x_d)[t]$ but not as an element in $\overline{\FF(x_1,\dots,x_d)}[t]$. Hence by Theorem~\ref{thm absolutely square-free equivalent conditions} $f$ has an irreducible factor in $\FF(x_1,\dots,x_d)[t^p]$. In particular the latter is not invertible, hence its degree in $t$ is not zero. By multiplying this factor by an element in $\FF(x_1,\dots,x_d)$ we obtain a factor of $f$ in $\FF[x_1,\dots,x_d][t^p]$. Let $g$ be this factor.
\end{proof}


\begin{thebibliography}{99}

\bibitem{Browning}
T.~D.~Browning, {\em Power-free values of polynomials}, Archiv der
Math. (2), 96 (2011), 139--150.

\bibitem{Dummit and Foote}
David S. Dummit and Richard M. Foote, {\em Abstract algebra}. Third
edition. John Wiley \& Sons Inc., Hoboken, NJ, 2004

\bibitem{Erdos}
P.~Erd\"{o}s. {\em  Arithmetical properties of polynomials}. J.
London Math. Soc. 28, (1953). 416--425.

\bibitem{Granville}
A.~Granville, {\em ABC  allows us to count square-frees}. Internat.
Math. Res. Notices 1998, no. 19, 991--1009.

\bibitem{HB}
D.R. Heath-Brown, Power-free values of polynomials, Quart. J. Math.,
64 (2013), 177--188.

\bibitem{Helfgott}
H.~Helfgott, {\em Power-free values, large deviations and integer
points on irrational curves}, J. Th\'eor. Nombres Bordeaux, 19
(2007), 433--472.

\bibitem{Hooley 1967}
C.~Hooley, {\em On the power free values of polynomials}.
Mathematika 14 1967 21--26.


\bibitem{Hooley77}
C.~Hooley, {\em On power-free numbers and polynomials} II, J. reine
angew. Math., 295 (1977), 1--21.

\bibitem{Kolchin}
E. R. Kolchin {\em diffrential algebra and algebraic groups}. Pure and Applied Mathematics, vol. 54. Academic Press, New York, 1973

\bibitem{Nair1}
M.~Nair, {\em Power free values of polynomials}. Mathematika, 23
(1976), 159--183.

\bibitem{Nair2}
M.~Nair, {\em Power free values of polynomials} II, Proc. London
Math. Soc., 38 (1979), 353--368

\bibitem{Poonen Duke}
B.~Poonen, {\em Squarefree values of multivariable polynomials}.
 Duke Math. J. 118 (2003), no. 2, 353--373.

\bibitem{Ramsay}
K.~Ramsay, {\em Square-free values of polynomials in one variable
over function fields}. Internat. Math. Res. Notices, no. 4 (1992)
97--102.

\bibitem{Ricci}
G. Ricci, {\em Ricerche aritmetiche sui polinomi}. Rend. Circ. Mat.
Palermo 57 (1933), 433--475.

\bibitem{Reuss}
T.~Reuss, {\em Power-Free Values of Polynomials}, arXiv:1307.2802
[math.NT]

\bibitem{Rudnick}
Z. Rudnick, {\em Square-free values of polynomials over the rational function field}, Journal of Number Theory, 135 (2014), 60--66

\bibitem{Schmidt}
 Wolfgang M.~Schmidt,
{\em Equations over finite fields: an elementary approach}. Second
edition. Kendrick Press, Heber City, UT, 2004
\end{thebibliography}
\end{document}